\documentclass[11pt, a4paper]{amsart}
\usepackage[english]{babel}
\usepackage[T1]{fontenc}
\usepackage[utf8]{inputenc}

\usepackage{geometry}
\usepackage[nodisplayskipstretch]{setspace}

\usepackage{amsrefs}
\DefineSimpleKey{bib}{primaryclass}{}
\DefineSimpleKey{bib}{archiveprefix}{}
\newcommand*{\bracketize}[1]{[#1]}
\BibSpec{arxiv}{%
  +{}{\PrintAuthors}{author}
  +{,}{ \textit}{title}
  +{}{ \parenthesize}{date}
  +{,}{ arXiv:}{eprint}
  +{}{ \bracketize}{primaryclass}
}

\usepackage{graphicx}

\usepackage{tikz}
\usetikzlibrary{positioning}
\usetikzlibrary{arrows.meta}
\usetikzlibrary{cd}
\usetikzlibrary{babel}

\usepackage[shortlabels]{enumitem}

\usepackage{amssymb}

\usepackage{dsfont}

\newcommand{\eqdef}{\mathrel{\mathop:}=}
\newcommand*{\ds}{\displaystyle}
\newcommand*{\aand}{\text{ and }}
\newcommand*{\oor}{\text{ or }}
\newcommand*{\aas}{\text{ as }}
\newcommand*{\wnv}{\text{whenever }}
\newcommand*{\fevery}{\text{ for every }}
\newcommand*{\pt}{\text{.}}
\newcommand*{\vg}{\text{, }}
\newcommand*{\pts}[1]{\left( #1 \right)}
\newcommand{\col}[1]{\left[ #1 \right]}
\newcommand{\roost}[1]{\left\{ #1 \right\}}
\newcommand{\bkt}[1]{\left\langle #1 \right\rangle}
\newcommand{\abs}[1]{\left\vert #1 \right\vert}
\newcommand{\nrm}[1]{\left\| #1 \right\|}
\newcommand{\floor}[1]{\lfloor #1 \rfloor}
\newcommand*{\id}{\mathrm{id}}
\newcommand*{\NN}{\mathds{N}}

\newcommand*{\RR}{\mathds{R}}
\newcommand*{\CC}{\mathds{C}}
\newcommand*{\II}{\mathds{I}}
\newcommand*{\US}{\mathds{S}^2}
\newcommand*{\UC}{\mathds{S}^1}
\newcommand{\set}[2]{\left\{ #1 \, : \, #2 \right\}}
\newcommand*{\compl}[1]{#1^{\mathsf{c}}}
\newcommand*{\restric}[2]{#1 |_{#2}}
\newcommand*{\cl}[1]{\overline{#1}}
\newcommand*{\fr}[1]{\partial #1}

\newcommand*{\odisk}[2]{\mathds{D}_{#2}\pts{#1}}
\newcommand*{\cdisk}[2]{\cl{\mathds{D}}_{#2}\pts{#1}}
\DeclareMathOperator{\inter}{int}
\DeclareMathOperator{\exter}{ext}
\DeclareMathOperator{\rel}{rel}
\newcommand*{\eps}{\varepsilon}
\newcommand*{\e}[1]{\mathrm{e}^{#1}}
\newcommand*{\seq}[2]{\pts{#1}_{#2}}
\newcommand*{\opint}[1]{\pts{#1}}
\newcommand*{\clint}[1]{\col{#1}}
\newcommand*{\clopint}[1]{\left[ \left. #1 \right) \right.}
\newcommand*{\opclint}[1]{\left( \left. #1 \right] \right.}
\newcommand*{\infdist}[3][]{d_{\infty}^{#1} \pts{#2,#3}}
\newcommand*{\dif}[2]{\mathrm{D} #1 \! \pts{#2}}
\newcommand*{\0}{\mathbf{0}}
\newcommand*{\1}{\mathbf{1}}
\newcommand*{\inff}{\mathbf{\infty}}
\newcommand*{\ii}{\mathrm{i} \,}
\newcommand*{\rad}[1]{\vec{r}_{#1}}
\DeclareMathOperator{\aangle}{ang}
\newcommand*{\ang}[1]{\aangle \! \pts{#1}}
\newcommand*{\iprod}[2]{\bkt{ #1 , #2 }}
\newcommand*{\ddx}{{}^{\partial}{\mskip -10mu \; / \, \mskip -8mu}_{\partial x}}
\newcommand*{\ddy}{{}^{\partial}{\mskip -10mu \; / \, \mskip -8mu}_{\partial y}}
\DeclareMathOperator{\diaag}{diag}
\newcommand*{\diag}[1]{\diaag \! \col{#1}}
\DeclareMathOperator{\homeos}{Homeo}
\newcommand*{\homeo}[2][]{\homeos_{\, #1} \! (#2)}
\DeclareMathOperator{\difeo}{Diff}
\newcommand*{\diff}[3][]{\difeo^{#2}_{\, #1} \! (#3)}
\DeclareMathOperator{\specorthgroup}{SO}
\newcommand*{\SO}[2]{\specorthgroup \! \pts{#1 ; #2}}
\DeclareMathOperator{\projspecgroup}{PSL}
\newcommand*{\PSL}[2]{\projspecgroup \! \pts{#1 ; #2}}
\DeclareMathOperator{\projspecunigroup}{PSU}
\newcommand*{\PSU}[2]{\projspecunigroup \! \pts{#1 ; #2}}
\DeclareMathOperator{\Mob}{M\ddot{o}b}
\newcommand*{\mob}[2][]{\Mob_{#1} (#2)}
\DeclareMathOperator{\Rot}{Rot}
\newcommand*{\rot}[2][]{\Rot (#2)}
\DeclareMathOperator{\Stab}{Stab}
\newcommand*{\stab}[2]{\Stab_{#2} \pts{#1}}
\newcommand*{\reg}[1]{{C}^{#1}}
\newcommand*{\cone}[1]{\mathcal{C}_{#1}}
\newcommand*{\omegalim}[2]{\omega_{#2} \pts{#1}}
\newcommand*{\alphalim}[2]{\alpha_{#2} \pts{#1}}
\newcommand*{\orb}[2][]{\gamma_{#1} \pts{#2}}

\newcommand*{\stman}[1][]{\mathrm{W}^{\mathrm{s}}_{#1}}
\newcommand*{\tanspace}[2]{T_{#1} #2 }
\newcommand*{\htop}[1]{h_{\text{top}} \pts{#1}}
\newcommand*{\gequiv}{\sim_{G}}
\newcommand*{\gtequiv}{\sim_{G_3}}
\newcommand*{\acc}[1]{\mathcal{A}_{#1}}
\newcommand*{\hp}{\mathcal{H}^{+}}
\newcommand*{\hm}{\mathcal{H}^{-}}
\DeclareMathOperator{\lftside}{L}
\newcommand*{\lft}[1]{\lftside \! \pts{#1}}
\DeclareMathOperator{\rgtside}{R}
\newcommand*{\rgt}[1]{\rgtside \! \pts{#1}}
\newcommand*{\mxy}[2]{M \! \col{#1,#2}}
\newcommand*{\mbar}[1]{\hat{M} \! \col{#1}}
\newcommand*{\mabc}[3]{\hat{M} \! \col{#1,#2,#3}}
\newcommand*{\T}[2]{T_{#1 #2}}
\newcommand*{\betap}{\beta^{+}}
\newcommand*{\betam}{\beta^{-}}
\newcommand*{\dtm}{\delta^{-}}
\newcommand*{\dtp}{\delta^{+}}

\newcommand*{\hphi}{{\phi}_0}
\newcommand*{\ttz}{\theta_0}
\newcommand*{\ttm}{\theta_M}
\newcommand*{\rhop}{\rho^{+}}
\newcommand*{\rhom}{\rho^{-}}
\newcommand*{\rhoz}{\rho_{0}}
\newcommand*{\Ah}{\hat{A}}
\newcommand*{\Cd}{C_{\delta}}
\newcommand*{\Dz}{\mathds{D}_{0}}
\newcommand*{\fd}{f_{\delta}}
\newcommand*{\gh}{\hat{g}}
\newcommand*{\hh}{\hat{h}}
\newcommand*{\hd}{h_{\delta}}
\newcommand*{\KK}{\mathcal{K}}
\newcommand*{\pd}{p_{\delta}}
\newcommand*{\pmin}{p_m}
\newcommand*{\pmax}{p_M}
\newcommand*{\Rz}{R_0}
\newcommand*{\sbar}{\bar{s}}

\newcommand{\tht}{\hat{t}}
\newcommand*{\tm}{t^{-}}
\newcommand*{\tp}{t^{+}} 
\newcommand*{\tbr}{\bar{t}}
\newcommand*{\ttil}{\tilde{t}}
\newcommand*{\uz}{u_0}
\newcommand*{\vz}{v_0}
\newcommand*{\xh}{\hat{x}}
\newcommand*{\yh}{\hat{y}}

\newcommand*{\zz}{z_0}
\newcommand*{\zh}{\hat{z}}
\newcommand*{\zhm}{\zh_{-}}
\newcommand*{\zhp}{\zh_{+}}
\newcommand*{\zt}{\tilde{z}}
\newcommand*{\zzt}{\tilde{z}_0}
\newcommand*{\zmin}{z_m}
\newcommand*{\zmax}{z_M}
\newcommand*{\wz}{w_0}
\newcommand*{\wh}{\hat{w}}
\renewcommand*{\wp}{w^{\perp}}

\newtheorem*{thma}{Theorem A}
\newtheorem*{thmb}{Theorem B}
\newtheorem*{thmc}{Theorem C}
\newtheorem*{thm}{Theorem}

\theoremstyle{definition}
\newtheorem{defi}{Definition}[section]
\newtheorem{lemma}[defi]{Lemma}
\newtheorem{cor}[defi]{Corollary}

\theoremstyle{remark}
\newtheorem{claim}{Claim}
\numberwithin{claim}{defi}
\newtheorem{auxclaim}{Claim}
\numberwithin{auxclaim}{subsection}
\newtheorem{case}{Case}

\usepackage{hyperref}
\hypersetup{ 
		colorlinks = true,
    linkcolor = blue, 
    urlcolor = blue, 
    citecolor = black 
}

\title{On proper extensions of the conformal group of sphere diffeomorphisms}

\author{Ulisses Lakatos de Mello}
\address{Instituto de Matemática e Estatística, Universidade de São Paulo, Rua do Matão 1010, Cidade Universitária, 05508-090 São Paulo, SP, Brazil}
\email{\href{mailto:lakatos@ime.usp.br}{lakatos@ime.usp.br}}
\urladdr{\href{http://www.ime.usp.br/~lakatos}{www.ime.usp.br/\(\sim\)lakatos}}
\thanks{U. Lakatos was partially supported by CNPq, Grant No. 159129/2015-0}

\author{Fábio Armando Tal}
\address{Instituto de Matemática e Estatística, Universidade de São Paulo, Rua do Matão 1010, Cidade Universitária, 05508-090 São Paulo, SP, Brazil}
\email{\href{mailto:fabiotal@ime.usp.br}{fabiotal@ime.usp.br}}
\thanks{F.A. Tal was supported by Fapesp and CNPq -- Brazil}

\begin{document}

\begin{abstract}
	In this paper, we prove that any group of diffeomorphisms acting on the 2-sphere and properly extending the conformal group of Möbius transformations must be at least 4-transitive or, more precisely, arc 4-transitive. In addition, we show that any such group must always contain an element of positive topological entropy. We also provide an elementary characterization, in terms of transitivity, of the Möbius transformations within the full group of sphere diffeomorphisms.
\end{abstract}

\maketitle

\tableofcontents

\setstretch{1.04}
 
\section{Introduction}

	Let \(M\) be a closed and oriented topological manifold, and consider the set of all its \emph{orientation preserving homeomorphisms}, denoted simply by \(\homeo{M}\). The usual \emph{uniform convergence metric} turns this set -- endowed with the composition operation \(\circ\) -- into a topological group, the subgroups of which one can try to understand and classify.
	
	In his early 2000s' essay \cite{Ghys01}, É. Ghys proposed such a classification scheme for closed and transitive groups acting on the unit circle \(\UC\). Here, closed refers to the uniform topology mentioned above, while \emph{transitive} means that any given point \(p\) can be mapped onto another given point \(q\) via some transformation in the group. The corresponding result, which we quote later on this paper, was proven in 2006 by J. Giblin and V. Markovic, see \cite{GiblinMarkovic06}.
	
	A relevant part of understanding the closed subgroups of \(\homeo{M}\) is to deal not only with their inclusions, but also with questions of \emph{maximality}. In other words, determining whether or not between a given subgroup and the full group of homeomorphisms one may find proper intermediate subgroups, up to their closures. For example, F. Le Roux proved in \cite{LeRoux14} that, in triangulable manifolds of dimension two or higher, the group of area preserving homeomorphisms is always maximal.
	
	Specializing in the unit sphere \(\US\), F. Kwakkel and the second author derived in \cite{TalKwakkel14} a number of results concerning important subgroups of \(\homeo{\US}\), one of which is \(\mob{\US}\), the \emph{Möbius group} of conformal 2-sphere diffeomorphisms. Among others, they left the question of whether the Möbius group is maximal within the full group of homeomorphisms, meaning that there would be no proper intermediate group of diffeomorphisms between \(\mob{\US}\) and \(\homeo{\US}\), up to uniform closure. \footnote{\cite{TalKwakkel14} is an unpublished work. Due to a difference of opinion between its authors, it remains as an unsubmitted preprint. Note that we reference here the v2 in arXiv, not the v3.  In any case, this note, being mostly self-contained, makes no direct use of any result in that work.}
	
	This question is known to have a positive answer in the case of the circle, and the proof there is related to higher orders of transitivity. In this paper we provide insight, from the transitivity viewpoint, into extensions of the Möbius group acting on the sphere, as stated below.
	
	\begin{thma} \label{thm:thma}
		Let \(G \subset \diff{1}{\US}\) be a group that is a proper extension of \(\mob{\US}\). Then, its identity component \(G_0\) is arc \(4\)-transitive. In particular, \(G\) is at least \(4\)-transitive.
	\end{thma}
	
	Above, the stronger concept of \emph{arc-transitivity} is introduced. Although more precisely defined in the text to follow, it essentially means that not only ordered lists \(\pts{p_1, \ldots, p_4}\) and \(\pts{q_1, \ldots, q_4}\) of four points can be mapped one onto another via a single transformation in the group \(G\), each \(q_i\) is actually the endpoint of \(p_i\)'s trajectory under a isotopy in \(G\) starting at the identity map.
	
	A straightforward application of Theorem A, using both the techniques developed in its proof and some basics of Nielsen-Thurston classification theory, is the following result.  It could also be derived with Theorem A and an abstract result from \cite{TalKwakkel14}).
	
	\begin{thmb} \label{thm:thmb}
			Let \(G \subset \diff{1}{\US}\) be a proper extension of \(\mob{\US}\). Then, \(G\) contains an element \(f\) fixing at least 4 points, and such that the restriction of \(f\) to the complement of these points is isotopic to a pseudo-Anosov map relative to these 4 points. In particular, \(f\) has strictly positive topological entropy.
	\end{thmb}
	
	Finally, we also include an interesting result -- albeit not new -- steaming from the same sort of techniques developed in the proofs of Theorems A and B. It is stated in terms of the general definition given below, which plays a key role in this paper.
	
	\begin{defi} \label{defi:trans}
		Given \(k \in \NN\), the action of a subgroup \(G \subset \homeo{M}\) is said to be \emph{\(k\)-transitive} if, for every pair of \(k\)-tuples \(\pts{p_1, \ldots, p_k}\) and \(\pts{q_1, \ldots, q_k}\), each of them composed by mutually distinct points, there exists some transformation \(g \in G\) such that \(q_i = g(p_i)\) for each \(i \in \roost{1, \ldots, k}\). When it holds, in addition, that for such a given $k$-tuple the only transformation fixing its every point is the identity map, the group is said to be \emph{sharply \(k\)-transitive}.
	\end{defi}	
	
	It is a widely known fact of complex analysis that \(\mob{\US}\) is sharply 3-transitive.  As it turns out, this is a \emph{defining property} of such group: in their 2014 work, Kwakkel and the second author show that any sharply \(3\)-transitive subgroup of \(\diff{1}{\US}\) extending the rotations group \(\rot{\US}\) must be \(\mob{\US}\). For completeness, and given that the cited work has not been published, we also present herein a new simple -- and in our opinion rather enjoyable -- direct proof of this fact, formally stated below.
	
	\begin{thmc} \label{thm:thmc}
		Let \(G \subset \diff{1}{\US}\) be a homogeneous group of diffeomorphisms. That is, \(G\) is an extension of the rotations group \(\rot{\US}\). Then, \(G\) sharply 3-transitive implies \(G=\mob{\US}$.
	\end{thmc}
	
\subsection{Preliminaries}

Let us start by agreeing upon terminology. As aforementioned, \(\homeo{M}\) denotes the set of orientation preserving homeomorphisms of a closed and oriented topological manifold \(M\), on which the uniform metric
	\begin{equation*}
		\infdist{f}{g} = \max_{p \in M} \roost{ d \pts{f(p), g(p)} }
	\end{equation*}
	is well defined. Above, \(d\) is some distance function generating the manifold's topology. If \(d\) comes from a Riemannian structure on \(M\), we may also consider the subset \(\diff{1}{M} \subset \homeo{M}\) of orientation preserving diffeomorphisms of class \(\reg{1}\) and its finer \emph{\(\reg{1}\) topology} -- or \emph{compact-open} -- which takes into account the local expressions of differentials as well.
	
	Consider the unit circle, \(M = \UC\). It can be thought of as the set of complex points at unit distance from the origin or as the compactification \(\RR \cup \roost{\infty}\). In the latter case, identification is provided by stereographic projection from the point \((0, 1)\). Each such description is linked to a canonical group acting on \(\UC\).
	
	In the first case, we have the \emph{rotations group}, \(\rot{\UC}\). This group is naturally identified with the circle itself, the rotation of angle \(\alpha\) being given as \(z \mapsto \e{2 \pi \ii \alpha} z\). It is thus a compact Lie group. In the second case, we have the \(\emph{Möbius group}\), \(\mob{\UC}\), obtained by transposing to the circle the action of \(\PSL{2}{\RR}\) on the extended real line by \emph{linear fractional transformations}.
	
	The rotations group can be realized as a proper subgroup of the Möbius group. Interestingly enough, these two groups yield a full description of the closed and transitive subgroups of \(\homeo{\UC}\), which is the content of the theorem by Giblin and Markovic.
	
	\begin{thm}[Giblin and Markovic, \cite{GiblinMarkovic06}]
		Let \(G\) be a closed and transitive subgroup of \(\homeo{\US}\) containing a nontrivial path connected component. Then, it is conjugate to one, and only one, of the following:
		\begin{enumerate}[(i)]
			\item \(\rot{\UC}\),
			\item \(\mob[k]{\UC}\),
			\item or \(\homeo[k]{\UC}\).
		\end{enumerate}
		Above, the subscript \(k \in \NN\) indicates cyclic covering of order \(k\). Furthermore, \(\mob{\UC}\) is maximal among the closed and transitive groups of \(\homeo{\UC}\). This means that there is no such a subgroup properly containing \(\mob{\UC}\) and properly contained in \(\homeo{\UC}\).
	\end{thm}
	
	Consider now the unit 2-sphere, \(M = \US\). It can be thought of either as the points of Euclidean 3-space \(\RR^3\) at unit distance from the origin or as the compactification \(\CC \cup \roost{\inff}\). In the latter case, identification is provided by stereographic projection from the North Pole \((0, 0, 1)\). Points on the plane and their stereographic images will be denoted by the same letters and confounded without notice. Accordingly, any sphere mapping that fixes \(\inff\) defines, by stereographic conjugation, a mapping of the plane with the same degree of regularity, also denoted by the same letter.
	
	The group of \emph{sphere rotations}, \(\rot{\US}\), consists of Euclidean isometries preserving \(\US\). It is realized by the orthogonal matrices \(\SO{3}{\RR}\), and is thus also a compact Lie group. Each such transformation amounts to prescribing an axis and a rotation angle around that axis. From these two facts, it can be seen that \(\rot{\US}\) is closed and transitive.
	
	Rotations are \emph{minimal} among transitive groups in the following sense: any compact subgroup \(G \subset \textrm{Homeo}(\US)\) is conjugate to a closed subgroup of \(\rot{\US}\) -- a result known to B. Kerékjártó since the 40s, and for which a proof according to contemporary standards is given by B. Kolev in \cite{Kolev06}. If \(G\) is also transitive, then it is conjugate to \(\rot{\US}\). The problem of classifying the groups contained between \(\rot{\US}\) and \(\homeo{\US}\) is thus called the \emph{kernel subgroup problem}, and such groups are hereafter called \emph{homogeneous}, following the terminology established in \cite{TalKwakkel14}.
	
	The \emph{Möbius group} \(\mob{\US}\) is an example of a homogeneous group. It is defined similarly to its circle counterpart, by the action of \(\PSL{2}{\CC}\) on the extended plane via linear fractional transformations. More precisely, one associates to (the class of) the matrix \(A\) a mapping \(M_A\)  as follows:
	\begin{equation*}
		\text{if } A = \begin{pmatrix} a & b \\ c & d \end{pmatrix} \in \PSL{2}{\CC} \text{ , then } M_A (z) = \frac{a \, z + b}{c \, z + d} \text{ for every } z \in \CC \cup \roost{\infty} \pt
	\end{equation*}
	
	This procedure not only induces a group of orientation preserving homeomorphisms of the sphere, it characterizes the \emph{conformal diffeomorphisms} of \(\US\) endowed with its canonical smooth structure. The Möbius group is thoroughly understood: each transformation has either one or two fixed points, and is conjugated to one out of four prototypical mappings -- called parabolic, elliptical, hyperbolic or loxodromic -- depending on the relation between the coefficients \(a\) and \(d\). In particular, rotations of the sphere are induced by the subgroup \(\PSU{2}{\CC}\). 
	
	These are classical results, some of them already known to Gauss, that can be found, for instance, in Chapter 3 of \cite{Needham00}. Kwakkel and the second author left open the question of whether the inclusion of \(\mob{\US}\) in \(\homeo{\US}\) is maximal. Our contribution to this problem was stated in \hyperref[thm:thma]{Theorem A} above. It should be stressed that, in the circle setting, 4-transitivity is sufficient to ensure \(k\)-transitivity for any \(k \geq 4\) (Theorem 6.4 in \cite{GiblinMarkovic06}).
	
	The constructions in this paper are mostly based upon \emph{isotopies}, here understood as jointly continuous maps \(f \colon I \times M \to M\), where \(I\) is a (possibly unbouded) real interval, such that the functions \(f_t = f(t, \!\cdot)\) are homeomorphisms of \(M\). We can either prescribe \(f\) or a family \(\seq{f_t}{t \in I}\) for which \(t \mapsto f_t\) is continuous with respect to \(d_{\infty}\). When the interval in question is the \emph{standard unit interval}, we denote it by \(\II = \clint{0,\!1}\).
	
	It will often be the case that \(f_t \in \diff{1}{M}\) for every \(t \in I\), but \(t \mapsto f_t\) can only be assured to be continuous with respect to \(d_{\infty}\). To avoid confusions, isotopies for which this association is actually continuous with respect to the \(\reg{1}\)-topology will be explicitly referred to as \emph{\(\reg{1}\)-isotopies}.
	
	\begin{defi} 
		Let \(G \subset \homeo{M}\). An isotopy \(\seq{f_t}{t \in I}\) such that \(0 \in I\), \(f_0 = \id\) and \(f_t \in G\) for every \(t \in I\) will be referred to as an \emph{\(\mathcal{I}G\)-isotopy}. By letting \(\sup I = + \infty\) and \(\inf I = - \infty\) whenever suitable, we then define:
		\begin{enumerate}[\bfseries (1)]
			\item The \emph{trajectory} of a point \(x \in M\) under \(f\) as \(\orb[f]{x} = \set{f_t (x)}{t \in I}\). Although this is \textit{a priori} just a set, it will often be thought of as a curve oriented according to its natural direction of travel.
			\item Its \(\omega\)-\emph{limit} as the following set of accumulation points:
			\begin{equation*}
				\omegalim{x}{f} = \set{y \in M}{ \text{ there exists some sequence } t_n \nearrow \sup I \text{ such that } f_{t_n}(x) \to y } \pt
			\end{equation*}
			A similar notion of \(\alpha\)-\emph{limit} is defined for sequences such that \(t_n \searrow \inf I\).
		\end{enumerate}
	\end{defi}
	
	We can now introduce the concept of arc transitivity, mentioned in \hyperref[thm:thma]{Theorem A}.
	
	\begin{defi} \label{defi:arctrans}
		Given \(k \in \NN\), a subgroup \(G \subset \homeo{M}\) is said to be \emph{arc \(k\)-transitive} if, for every pair of \(k\)-tuples as in \hyperref[defi:trans]{Definition} \ref{defi:trans}, there exists an \(\mathcal{I}G\)-isotopy \(\seq{g_t}{t \in \II}\) such that \(g_1(p_i) = q_i\) for each \(i \in \roost{1, \ldots, k}\).
	\end{defi}	
	
	The suggestive terminologies above are useful in the developments that follow. Although analogies with flows are limited, the following establishes some parallels.
	
	\begin{defi} \label{defi:gequiv}
		For a fixed a subgroup \(G \subset \homeo{M}\) and given \(z,w \in M\), say that
		\begin{equation*} 
			z \gequiv w \iff \text{ there exists an \(\mathcal{I}G\)-isotopy \(\seq{f_t}{t \in \II}\) in \(G\) such that \(f_1 (z) = w\).}
		\end{equation*}
		This is an equivalence relation, under which the class of a point \(z \in M\) is denoted by \(\acc{z}\), and referred to as the \emph{points accessible from \(z\)} (in \(G\)).
	\end{defi}
	
	\begin{lemma} \label{lemma:equivcross}
		Let \(\seq{f_t}{t \in I}\) and \(\seq{h_t}{t \in J}\) be \(\mathcal{I}G\)-isotopies, where \(I\) and \(J\) are intervals of any kind. Then, for any two \(z,w \in M\) such that \(\orb[f]{z} \cap \orb[h]{w} \neq \emptyset\) we have that \(\acc{w} = \acc{z}\).
	\end{lemma}
	\begin{proof}
		Say that \(f_a (z) = h_b (w)\), for some \(a \in I\), \(b \in J\). Given \(y \in \acc{w}\), there exists an \(\mathcal{I}G\)-isotopy \(\seq{g_t}{t \in \II}\) such that \(g_1 (w) = y\). Then,
		\[
			k_t = \begin{cases}
				f_{3a t} & \text{if } 0 \leq t \leq \dfrac{1}{3} \vg \\[0.75em]
				h_{\pts{3t - 2} b} \circ h_b^{-1} \circ f_a & \text{if } \dfrac{1}{3} \leq t \leq \dfrac{2}{3} \vg \\[0.75em]
				g_{3t-2} \circ h_b^{-1} \circ f_a & \text{if } \dfrac{2}{3} \leq t \leq 1 \vg
			\end{cases}
		\]
		is an \(\mathcal{I}G\)-isotopy satisfying \(k_1 (z) = y\). This shows that \(\acc{w} \subset \acc{z}\), and the converse inclusion follows by symmetry.
	\end{proof}
	
	We finish this introduction by fixing some reference points that play a key role in the arguments that follow:
	\begin{itemize}
		\item \(\inff\) is the North Pole,
		\item \(\0\) is the South Pole, and corresponds to the plane's origin,
		\item \(\1\) is the point \((1,0,0)\), which corresponds to its counterpart on the real axis.
	\end{itemize}  
	
	The meridian through \(\0\), \(\1\) and \(\inff\) -- which is the stereographic image of the \(x\)-axis (or real axis) -- is denoted by \(\Gamma\). This meridian defines on its left a \emph{western hemisphere} \(\hp\), corresponding to the upper half-plane, and on its right an \emph{eastern hemisphere} \(\hm\), corresponding to the lower half-plane. For any group \(G\) we also adopt the following notations for stabilizers:
	\begin{equation*}
		 G_1 = \stab{\inff}{G} \; \vg \; G_2 = \stab{\0, \inff}{G} \; \aand \; G_3 = \stab{\0, \1, \inff}{G} \pt
	\end{equation*}
	
\subsection{Paper outline}

	Given a group \(G\) extending \(\mob{\US}\), we consider the subgroups \(G_k\), \(1 \leq k \leq 3\), as above. Our objective is to conclude that \(G_3\) is (one) transitive.
	
	We begin by establishing an \hyperref[lemma:extlemma]{Extension Lemma} at the end of \hyperref[sec:ext]{Section \ref{sec:ext}}, which states that \(G_2\) must contain an isotopy between the identity and a map having a hyperbolic saddle point at the (fixed) South Pole. This is achieved by starting with a nonconformal map and continuously parametrizing rotations and homotheties, both of which are Möbius transformations, in such a way as to create eigendirections and modulate the corresponding eigenvalues. 
	
	From this starting point, we fix the  privileged reference meridian \(\Gamma\) and promote two parallel processes. In \hyperref[sec:fund]{Section \ref{sec:fund}}, we prove that points outside of it admit full time isotopies in \(G_3\) accumulating at the poles. In \hyperref[sec:cross]{Section \ref{sec:cross}}, we prove that there is a finite time isotopy in \(G_3\) for which some point on this meridian starts at one side of it and ends at the other side.
	
	In \hyperref[sec:closing]{Section \ref{sec:closing}} we show our main theorems. \hyperref[thm:thma]{Theorem A} is derived upon combining isotopies of the types described and concluding that all but three points of the sphere are actually arc connected in the sense of \hyperref[defi:arctrans]{Definition \ref{defi:arctrans}}, which yields arc transitivity of \(G_3\). \hyperref[thm:thmb]{Theorem B} is also derived by convenient combinations of segments of such isotopies, but to produce a ``topological figure 8'', a device that implies positive entropy due to the Nielsen-Thurston classification theory.
	
	Lastly, in \hyperref[sec:appendix]{Section \ref{sec:appendix}} we derive \hyperref[thm:thmc]{Theorem C}. First, a purely topological argument shows that, if \(G\) is a sharply \(3\)-transitive and homogeneous group of homeomorphisms, then \(G_2\) must permute parallels. This fact, when combined with differentiability, yields conformality -- first at the poles, and then at every point.
	
\section{Extensions of Möbius} \label{sec:ext}

	\begin{defi}
		For a given \(0 \leq \theta < 2 \pi\), \(R_{\theta}\) denotes the counterclockwise rotation of angle \(\theta\)
		\begin{itemize}
			\item around the \(z\)-axis, if it is thought of as a mapping of the 2-sphere,
			\item around the origin, if it is thought of as a planar mapping. In this case, it may be written as \(R_{\theta} (z) = \e{2 \pi \ii \theta} \, z\).
		\end{itemize}
		Also, for a given \(\rho > 0\), we let \(H_{\rho} \in \mob[2]{\US}\) denote the transformation induced by the plane homothety \(z \mapsto \rho z\). Since the tangent space \(\tanspace{\0}{\US}\) can be identified with the horizontal subspace \(\RR^2 \times \roost{0} \simeq \RR^2\), generated by the canonical basis \(\{\ddx, \ddy\}\), we may write, in a slight abuse of notation:
	\[
		\dif{R_{\theta}}{\0} = R_{\theta} \quad \aand \quad \dif{H_{\rho}}{\0} = \rho \, \id \, \pt
	\]
	\end{defi}
	
	We remark that, if \(u,v\) are nonzero vectors, the \emph{angle} \(0 \leq \ang{z,w} \leq \pi\) between them is defined in the usual way, as \(\arccos \col{ \iprod{u}{v} / \pts{ \abs{u} \abs{v} } }\).
	
	\begin{lemma} \label{lemma:saddle}
		Let \(G \subset \diff{1}{\US}\) be a proper extension of \(\mob{\US}\). Then, there exists \(\gh \in G_2\) for which \(\0\) is a hyperbolic saddle point. More precisely, \(\dif{\gh}{\0} = \diag{\lambda, \lambda^{-1}}\) with respect to the canonical basis, where \(0 <\lambda < 1\).
	\end{lemma}
	\begin{proof}
		Since \(G\) is a proper extension of the conformal group, it contains at least one mapping \(\hh\) which is nonconformal at some point of the sphere. By precomposing and postcomposing with suitable Möbius transformations, it can be assumed that such map lies in \(G_2\) (that is, it fixes the poles) and that \(A = \dif{\hh}{\0}\) is nonconformal. 
		 
		This means that there exists a pair of unit vectors \(u,v\) such that \(\ang{Au, Av} \neq \ang{u,v}\). Let \(R_{(1)}\) be a planar rotation such that \(R_{(1)} \pts{Au}\) is a positive multiple of \(u\). Then, \(A_1 = R_{(1)} A\) is a linear mapping having at least one eigenvalue \(\lambda_1 > 0\), for which \(u\) is an unit eigenvector. It cannot be the case that \(\lambda_1\) is of maximal geometric multiplicity, for that would imply \(A\) conformal. This leaves two possibilities.
		
		\begin{case}
			\(A_1\) is a defective matrix.
		\end{case}	
			\noindent Then, we may fix an orthogonal chain of generalized eigenvectors \(\roost{u,w}\). That is, \(A_1 w = u + \lambda_1 w\) and \(\iprod{u}{w} = 0\). For each \(\phi \in \clint{0,\!\pi/2}\) we define
				\[
					x_{\phi} = \cos \phi \; u + \sin \phi \; w / \abs{w} \; \aand \; y_{\phi} = - \sin \phi \; u + \cos \phi \; w / \abs{w} \vg
				\]
				obtaining an orthonormal frame \(\roost{x_{\phi}, y_{\phi}}\) of \(\tanspace{\0}{\US}\). Consider now the continuous function
				\[
					\xi (\phi) = \iprod{A_1 x_{\phi}}{A_1 y_{\phi}} / \abs{ A_1 x_{\phi} } \abs{ A_1 y_{\phi} } \pt 
				\]
			It satisfies \(\xi (0) \phantom{\cdot} \xi (\pi/2) < 0\). We thus have \(\xi (\hphi) = 0\) for some \(0 < \hphi < \pi/2\). This means that, if \(\xh = x_{\hphi}\) and \(\yh = x_{\hphi}\), then the orthonormal frame \(\roost{ A_1 \xh / | A_1 \xh |  , A_1 \yh / \abs{ A_1 \yh } }\) can be applied onto \(\roost{\xh, \yh}\) by a certain planar rotation \(R_{(2)}\), because \(A_1\) preserves orientation. Thus, \(R_{(2)} \pts{A_1 \xh} = \nu_1 \, \xh\) and \(R_{(2)} \pts{A_1 \xh} = \nu_2 \, \yh\), where \(\nu_1, \nu_2 > 0\).
				
				Once again, we must have \(\nu_1 \neq \nu_2\), or \(A\) would be conformal. Let us assume, without loss of generality, that \(0 < \nu_1 < \nu_2\). By defining \(A_2 = R_{(2)} A_1\), \(\rho = 1 / \sqrt{\nu_1 \nu_2}\) and \(\lambda = \sqrt{\nu_1 / \nu_2} < 1\), we have \(\pts{\rho \, A_2} \xh = \lambda \xh\) and \(\pts{\rho \, A_2} \yh = \lambda^{-1} \, \yh\).
				Lastly, let \(R\) be a planar rotation such that \(R^{-1} \, \xh = \partial / \partial x\). Then, \(R^{-1} \, \yh = \eps \; \partial / \partial y\), where \(\eps = 1\) if \(\roost{\xh, \yh}\) is a positive basis, or \(\eps = -1\) otherwise. Defining \(\Ah = \rho \, R^{-1} A_2 R\): 
				\begin{gather*}
					\Ah \pts{\ddx} = \lambda \; \ddx \quad \aand \quad \Ah \pts{\ddy}  = \eps^2 \, \lambda^{-1} \; \ddy = \lambda^{-1} \; \ddy \, \pt
				\end{gather*}
				
				Therefore, \(\hat{A}\) has the form described in the Lemma's statement with respect to the canonical basis. But, if \(\gh = R^{-1} \circ H_{\rho} \circ R_{(2)} \circ R_{(1)} \circ \hh \circ R \in G_2\), then \(\dif{\gh}{\0} = \Ah\), completing the proof in this case.
				
			\begin{case}
			\(A_1\) has a second eigenvalue \(\lambda_2 \neq \lambda_1\)
		\end{case}
			\noindent Fix a positive basis of unit eigenvectors \(\roost{u,w}\). Let \(\wp\) be the orthogonal complement of \(w\) with respect to \(u\). For each \(\phi \in \clint{0, \pi/2}\), we define an orthonormal frame \(\roost{x_{\phi}, y_{\phi}}\) as
				\[
					x_{\phi} = \cos \phi \; u + \sin \phi \; \wp / |\wp| \quad \aand \quad y_{\phi} = - \sin \phi \; u + \cos \phi \; \wp / |\wp| 	\vg	
				\]
				and consider the same continuous function \(\xi\) as before. We now have \(\xi (0) \phantom{\cdot} \xi (\pi/2) \leq 0\), so there exists \(0 \leq \hphi < \pi/2\) such that \(A_1 \xh\) and \(A_2 \yh\) are orthogonal, where \(\xh = x_{\hphi}\) and \(\yh = y_{\hphi}\). We may thus fix a planar rotation \(R_{(2)}\) -- possibly the identity -- applying the orthonormal frame \(\roost{ A_1 \xh / | A_1 \xh | , A_1 \yh / | A_1 \yh | }\) onto \(\roost{\xh, \yh}\). From here, the proof follows as in Case 1, by defining \(\nu_i\), \(\rho\), \(\lambda\) and \(R\) exactly as before.
	\end{proof}
	
	\begin{lemma}[The Extension Lemma] \label{lemma:extlemma}
		Let \(G \subset \diff{1}{\US}\) be a proper extension of \(\mob{\US}\). Then, there exists an \(\mathcal{I}G_2\) \(\reg{1}\)-isotopy \(\seq{g_t}{t \in \II}\) such that:
		\begin{enumerate}[\bfseries (1)]
			\item for every \(t > 0\), the differential \(\dif{g_t}{\0}\) is a hyperbolic saddle, having the tangent line \(\tanspace{\0}{\Gamma}\) as its stable direction,
			\item \(\dif{g_1}{\0} = \diag{\lambda, \lambda^{-1}}\) with respect to the canonical basis of \(\tanspace{\0}{\US}\).
		\end{enumerate}				
	\end{lemma}
	\begin{proof}
		By \hyperref[lemma:saddle]{Lemma \ref{lemma:saddle}}, we may fix \(\gh \in G_2\) such that \(\Ah = \dif{\gh}{\0} = \diag{\mu, \mu^{-1}}\), where \(0 < \mu < 1\). For \(s \in \clint{0, \pi/2}\), let \(B_s = \Ah^{-1} \, R_s \, \Ah\) and \(v_s = B_s ( \ddx )\). Then, \(\abs{v_s} < 1\) for \(0 < s \leq \pi/2\). Next, we define the continuous function \(\theta \colon \clint{0, \pi/2} \to \RR\) to be the angle between \(v_s\) and \(\ddx\), and let \(C_s = R_{- \theta(s)} \, B_s\). Then, \(C_0 = \id\) and, for \(0 < s \leq \pi/2\), the \(x\)-axis is a contracting direction for \(C_s\), of rate \(\lambda_s = \abs{v_s}\). But, since \(\det C_s = 1\), it must also have an invariant direction of expansion rate \(\lambda_s^{-1}\). In other words, \(C_s\) is a hyperbolic saddle, although not necessarily diagonal. When \(s = \pi/2\), however, it can be computed directly that \(C_{\pi/2} = \diag{\lambda, \lambda^{-1}}\), where \(\lambda = \mu^2\). Lastly, we define \(\seq{g_t}{t \in \II}\) as:
		\begin{equation} \label{eq:isotopy} 
			g_t = R_{- \theta \pts{\frac{\pi t}{2}} } \circ \gh^{-1} \circ R_{\frac{\pi t}{2}} \circ \gh \in G_2 \pt
		\end{equation}				
		Since \(\theta \pts{\cdot}\) is a continuous function, \eqref{eq:isotopy} defines an \(\mathcal{I}G_2\)-isotopy which is continuous not only with respect to the uniform topology, but also to the \(C^1\) topology. For each \(t\), we have \(\dif{g_t}{\0} = C_{{\pi t}/{2}}\). Thus, given that the \(x\)-axis corresponds to \(\Gamma\) on the sphere -- with the direction generated by \(\ddx\) identified with \(\tanspace{\0}{\Gamma}\) -- the considerations previously made translate into the statements of the lemma, completing the proof.
	\end{proof}
	
\section{A Fundamental Lemma} \label{sec:fund}

	For each angle \(\alpha \in \opint{0, \!\pi/2} \), the subset of the plane describe in polar coordinates as follows shall be referred to as the \emph{\(\alpha\)-cone}:
	\[
		\cone{\alpha} = \set{ r \e{ \ii \phantom{\cdot} \theta } }{r \geq 0 \aand \text{either } \abs{\theta} \leq \alpha \oor \abs{\theta - \pi} \leq \alpha } \pt
	\]	
		
	We remark that, for \(0 < \lambda < 1\), any cone  is ``broadened'' under the action of the hyperbolic matrix \(A = \diag{\lambda, \lambda^{-1}}\). More precisely, consider \(\UC\) parametrized by the counterclockwise	polar angle \(\theta\) that each direction of space makes with the \(x\)-axis. Then, the action of \(A\) induces a monotone circle dynamics \(\tilde{A}\) -- by radial projection -- such that \(\tilde{A}\) has repelling fixed points at \(0\) and \(\pi\) and attracting fixed points at \(\pi/2\) and \(3\pi / 2\).
		
	Now, if we let \(g\) be a planar diffeomorphism fixing the origin, then, for a given \(0 < \eps < \pi\), differentiability yields \(\delta > 0\) such that
\(\ang{\, g(z),  \dif{g}{\0} \! z \, } \leq \eps\) whenever \(0 < \abs{z} < \delta\). The next lemma is a consequence of these two remarks, along with the fact that the following diagram commutes:
	\[
		\begin{tikzcd}[column sep=1.25cm]
			v \in \RR^2 \setminus \roost{\0} \arrow[d, "\text{ polar angle }" left, mapsto] \arrow[r, "\text{saddle}", mapsto] & Av \arrow[r, "\text{rotation}", mapsto] & R_{\pm \tau} \pts{Av} \arrow[d, "\text{ polar angle }", mapsto] \\
			\theta \in \UC \arrow[r, "\text{source-sink}" below, mapsto] & \tilde{A} \pts{\theta} \arrow[r, "\text{translation}" below, mapsto] & \tilde{A} \pts{\theta} \pm \tau 		
		\end{tikzcd} 
	\]		
		
	\begin{lemma} \label{lemma:projsep}
		Let \(g\) be a planar diffeomorphism for which the origin is a hyperbolic fixed point satisfying \( \dif{g}{\0} = \diag{\lambda, \lambda^{-1}}\), \(0 < \lambda < 1\). Then, for a given \(0 < \alpha < \pi / 2\), there exist \(\tau > 0\) and \(\delta > 0\) such that:
		\begin{equation*}
			0 < \abs{z} < \delta \; \aand \; z \notin \cone{\alpha} \;  \text{ imply } \; R_{\omega} \pts{g(z)} \notin \cone{\alpha} \, \vg \wnv \, \abs{\omega} < \tau \pt
		\end{equation*}
		In particular, \(g^k (z) \notin \cone{\alpha}\) for every \(k \in \NN\) such that the orbit \(\roost{z, g(z), \ldots, g^{k-1}(z)}\) remains in \(\odisk{\0}{\delta}\). 		
	\end{lemma}
	
	We keep this result aside for now and move on to understand how isotopies of the kind defined in the \hyperref[lemma:extlemma]{Extension Lemma} act on cones. For that, we record the following consequence of the chord length formula from elementary geometry: 
	\begin{equation} \label{eq:polardist}
		\abs{z - w} \geq \min \roost{ \abs{z}, \abs{w} } \sin \col{ \frac{\ang{z,w}}{2} } \pt
	\end{equation}	
	
	\begin{lemma} \label{lemma:precone}
		Let \(\seq{g_t}{t \in \II}\) be a planar \(\reg{1}\)-isotopy such that the origin is a fixed point and \(\dif{g_t}{\0}\) has the \(x\)-axis as an invariant direction for every \(t\). Then, given \(0 < \alpha <\pi/2\), there exist \(0 < \beta < \alpha \) and \( \rho > 0\) such that:
		\begin{equation*}
			z \in \odisk{\0}{\rho} \: \aand \: z \notin \cone{\alpha} \; \text{ imply } \; g_t(z) \notin \cone{\beta} \: \fevery t \in \II \pt
		\end{equation*}
	\end{lemma}
	\begin{proof}
		For a fixed \(t \in \II\), let \(A_t = \dif{g_t}{\0}\). Consider \(v_{\alpha}\) and \(v_{\alpha}^{\ast}\) unit vectors of angle \(\alpha\) and \(2\pi - \alpha\), respectively, whose spans delimit \(\cone{\alpha}\). Given that the \(x\)-axis is \(A_t\)-invariant, \(A_t v_\alpha\) lies to the left and \(A_t v_{\alpha}^{\ast}\) lies to the right of it. Also, at least one among \(\ang{A_t v_{\alpha}, \ddx}\) and \(\ang{A_t v_{\alpha}^{\ast}, \ddx}\) has to be smaller than \(\pi / 2\), because \(A_t\) is an orientation preserving linear isomorphism. Thus, by taking
		\[
			\beta_t = \min \roost{ \alpha \, , \, \ang{A_t v_{\alpha}, \ddx} \, , \, \ang{A_t v_{\alpha}^{\ast}, \ddx} } 
		\]
		we have \(0 < \beta_t \leq \alpha\) and \(A_t v \notin \cone{\beta_t}\) whenever \(v \notin \cone{\alpha}\). This reasoning yields a global solution to the associated linear problem. Once it is done, let \(\eps_t = r_t \, c_t \, \sin \pts{ {\beta_t}/{4} }\), where \(r_t = \pts{1 + \sin \pts{ \beta_t / 4 }}^{-1}\) and \(c_t > 0\) is such that \(\abs{A_t v} \geq c_t \abs{v}\) for every vector \(v\). Having these choices in mind, we obtain \(\rho_t > 0\) such that:
		\begin{equation} \label{eq:differentiability}
			\abs{g_t (z) - A_t z} \leq \frac{\eps_t}{2} \abs{z} \: \wnv \: 0 < \abs{z} < \rho_t \pt
		\end{equation}		
		With respect to the compact-open topology, we consider the following \emph{subbasic neighbourhood} \(\mathcal{V}_t\) of \(g_t\), as in the first chapter of \cite{Banyaga97}:
		\begin{equation*}
			\mathcal{V}_t = \set{f \in \diff{1}{\RR^2}}{ \sup_{\cdisk{\0}{\rho_t}} \abs{ f - g_t } \leq \frac{\eps_t}{2} \; \aand \; \sup_{\cdisk{\0}{\rho_t}} \nrm{ \mathrm{D}f - \mathrm{D}g_t } \leq \frac{\eps_t}{2} } \pt
		\end{equation*}		
		By continuity of the isotopy, there exists \(\delta_t > 0\) with the following property:
		\begin{equation} \label{eq:sbasicnbd}
			s \in \II \aand \abs{s - t} < \delta_t \; \text{ imply } \; g_s \in \mathcal{V}_t \pt 
		\end{equation}		
		Thus, when \(\abs{s - t} < \delta_t\) and \(\abs{z} < \rho_t\) simultaneously, we have
		\begin{align*}
			\abs{ g_s (z)  - A_t z } &\leq \abs{ g_s(z) - g_t(z) } + \abs{ g_t (z) - A_t z } \\
			&\leq \pts{ \sup_{\cdisk{\0}{\rho_t}} \nrm{ \mathrm{D} (g_s - g_t) } } \abs{z} + \frac{\eps_t}{2} \abs{z} \leq \eps_t \abs{z} \vg
		\end{align*}
		where the Mean Value Inequality was used, along with \eqref{eq:differentiability} and \eqref{eq:sbasicnbd}. It follows that \(\abs{g_s (z)} \geq \pts{ c_t - \eps_t } \abs{z}\), and the right hand side of this inequality is strictly positive. This allows us to conclude that \(g_s(z)\) lies in a closed disk centered at \(A_tz\) and of radius \(\eps_t \abs{z}\), to which the origin is an external point. In particular, \(\ang{g_s (z), A_t z} < \pi / 2 \). These facts, along with \eqref{eq:polardist}, yield that
		\[
			\sin \col{ \frac{ \ang{g_s(z), A_t z} }{2} } \leq \frac{ \abs{ g_s (z)  - A_t z } }{ \pts{ c_t - \eps_t } \abs{z} } \leq \frac{ \eps_t \abs{z} }{ \pts{ c_t - \eps_t } \abs{z} } = \sin \pts{\frac{\beta_t}{4}} \pt
		\]	
		
		These calculations allow us to conclude that \( \ang{g_s(z), A_t z} \leq \beta_t / 2 \), as long as \(\abs{s-t} < \delta_t\) and \(0< \abs{z} < \rho_t\). If also \(z \notin \cone{\alpha}\), we know from the linear case that \(A_t z \notin \cone{\beta_t}\). It follows that \(g_s (z) \notin \cone{\beta_t / 2}\). Lastly, by compactness we may find \(t_1, \ldots, t_n\) such that \(\II = \bigcup_{j=1}^{n} \opint{t_j - \delta_{t_j} , t_j + \delta_{t_j} } \cap \II\). Thus, by taking \(\rho = \min_j \rho_{t_j}\) and \(\beta = \min_j \beta_{t_j} / 2 \), the desired conditions are satisfied.
	\end{proof}
	
	\begin{cor} \label{lemma:conelemma}
		Let \(\seq{g_t}{t \in \II}\) be a planar \(\reg{1}\)-isotopy such that the origin is a fixed point and \(\dif{g_t}{\0}\) has the \(x\)-axis as an invariant direction for every \(t\). Then, given \(0 < \alpha < \pi/2\), there exist \(0 < \betam < \betap < \alpha \) and \( \rho > 0\) such that:
		\begin{enumerate}[\bfseries (1)]
			\item \(z \in \odisk{\0}{\rho}\) and \(z \notin \cone{\alpha}\) imply \(g_t(z) \notin \cone{\betap}\) for every \(t \in \II\),
			\item \(z \in \odisk{\0}{\rho}\) and \(z \in \cone{\betam}\) imply \(g_t(z) \in \cone{\betap}\) for every \(t \in \II\).
		\end{enumerate}
	\end{cor}
	\begin{proof}
		Given \(\alpha\), \hyperref[lemma:precone]{Lemma \ref{lemma:precone}} yields \(0 < \betap < \alpha\) and \(\rhop > 0\) such that \(g_t(z) \notin \cone{\betap}\) for every \(t \in \II\), whenever \(z \in \odisk{\0}{\rhop}\) and \(z \notin \cone{\alpha}\). But, looking at the isotopy \(\seq{g_{t}^{-1}}{t \in \II}\) we see that it satisfies exactly the same hypotheses as listed in that Lemma. So, for this isotopy and the angle \(\betap\) just encountered we obtain \(\rhom > 0\) and \(0 < \betam < \betap\) such that:
		\[
			w \in \odisk{\0}{\rhom} \: \aand \: w \notin \cone{\betap} \; \text{ imply } \; g_{t}^{-1} (w) \notin \cone{\betam} \fevery t \in \II \pt
		\]
		Let \(\eta > 0\) be such that \(\abs{g_t (z)} < \rhom\) for every \(t \in \II\), whenever \(\abs{z} < \eta\). Then, by setting \(\rho\) as \(\min \{ \eta, \rhop \}\) we have the proposed statements satisfied. Indeed, if that was not the case,  \(g_s (z) \notin \cone{\betap}\) for some \(s \in \II\) and \(z \in \odisk{\0}{\rho} \cap \cone{\betam}\) would lead to a contradiction.
	\end{proof}
	
	\begin{defi} \label{defi:mbar}
		Given finite and nonzero points \(z,w\), we let \(\mxy{z}{w} \in \mob[2]{\US} \) be the \emph{unique} Möbius transformation fixing the poles and mapping \(z\) to \(w\). Also, we denote \(\mbar{z} =\mxy{z}{\1}\).
	\end{defi}
	
	We can write down explicit formulae for these transformations and see that \(\pts{z,\!w} \mapsto \mxy{z}{w}\) is continuous. Also, if \(\KK \subset \US\) is a nonempty compact set bounded away from \(\0\), 
	\begin{equation} \label{eq:auxlim}
		\text{ the sets \(\mbar{x} \pts{\KK}\) converge to \(\roost{\inff}\) on the Hausdorff distance as \(x \to \0\). }
	\end{equation}		
	Having settled these notations and technical results, we are now able to prove our Fundamental Lemma.
	
	\begin{lemma}[The Fundamental Lemma] \label{lemma:fundamentallemma}
		Let \(G \subset \diff{1}{\US}\) be a proper extension of \(\mob{\US}\). Then, for a given point \(\zz\) \emph{not} on the meridian \(\Gamma\) there exists an \(\mathcal{I}G_3\)-isotopy \(\seq{I^{\zz}_{t}}{t \geq 0} \), \emph{depending on \(\zz\)}, such that:
		\begin{enumerate}[\bfseries (1)]
			\item the trajectory of \(\zz\) under \(I^{\zz}\) does not intercept \(\Gamma\), and
			\item the \(\omega\)-limit of \(\zz\) satistfies \(\omegalim{ \zz }{ I^{\zz} } = \roost{\inff}\).
		\end{enumerate}
	\end{lemma}
	\begin{proof}
		Given \(\zz \notin \Gamma\), we assume for concreteness that it lies on the upper half-plane, and is thus given in polar coordinates as \(\zz = \Rz \,  \e{\ii \ttz}\), \(0 < \ttz < \pi\).
		
		Let \(\seq{g_t}{t \in \II}\) be as in the \hyperref[lemma:extlemma]{Extension Lemma}. Since \(\inff\) is fixed throughout, it can be thought of as a planar \(\mathcal{I}G_1\)-isotopy such that \(\dif{g_t}{\0}\) has the \(x\)-axis as an invariant direction for every \(t\) and \(\dif{g_1}{\0} = \diag{\lambda, \lambda^{-1}}\), \(0 < \lambda < 1\). To ease notation, we write \(g \eqdef g_1\) and \(A \eqdef \dif{g}{\0}\).
		
		Fix some \(0 < \alpha < \pi / 2\) such that the direction through \(\ttz\) is external to \(\cone{2 \alpha}\). \emph{With respect to  \(\alpha\)}, let \(\delta > 0\) and \(\tau > 0\) be as described in \hyperref[lemma:projsep]{Lemma \ref{lemma:projsep}}. Regarding this same \(\alpha\), and also the isotopy \(\seq{g_t}{t \in \II}\), \hyperref[lemma:conelemma]{Corollary \ref{lemma:conelemma}} yields a radius \(\rho > 0 \) and angles \(0 < \betam < \betap < \alpha\) such that, for \(z \in \odisk{\0}{\rho}\) and every \(t \in \II\), \(z \notin \cone{\alpha}\) implies \(g_t(z) \notin \cone{\betap}\), whilst \(z \in \cone{\betam}\) implies \(g_t(z) \in \cone{\betap}\). Lastly, we characterize the stable manifold \(\stman\) of \(g\) at \(\0\). Since the stable direction of \(A\) is the \(x\)-axis, \(\stman\) may be assumed to be a Lipschitz graph of the form \(y = y(x)\) having horizontal tangent at the origin. So, for \(\tau\) and \(\betam\) as obtained above and a sufficiently small radius \(\sigma > 0\), it can be assumed that \(\stman \cap \odisk{\0}{\sigma} \subset \cone{\min \roost{\tau \, , \, \betam}}\).
		
		Let \( 	0 < \rhoz < \min \roost{\delta, \rho, \sigma, 1} \). Then, in the disk \(\Dz \eqdef \odisk{\0}{\rhoz}\) all of the conditions described above are mutually satisfied, as conveyed in \hyperref[fig:d0]{Figure \ref{fig:d0}}.
		
		\begin{figure}[htb]
			\centering
			\includegraphics[scale=1.1]{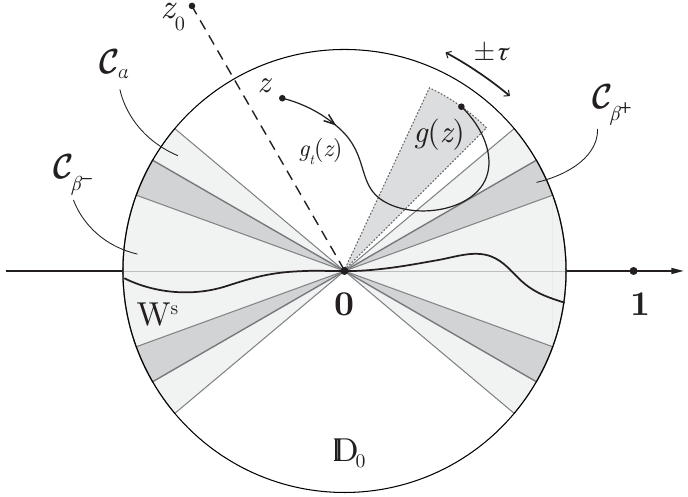}
			\caption{Inside the disk \(\Dz\), the isotopy path under \(g_t\) of points \(z\) outside of the \(\alpha\)-cone never enters the \(\betap\)-cone, while the stable manifold of \(g\) is a Lipschitz graph fully contained within the \(\min\roost{\tau, \betam}\)-cone.}
			\label{fig:d0}
		\end{figure}
		
		Once these choices are made, fix a positive real number \(0 < r_1 < \min \roost{ \rhoz \, , \, \rhoz/ \Rz }\)	and let \(\rho_1 = r_1 R_0\) and \(\KK_0 = \cl{\fr{\odisk{\0}{\rho_1}} \setminus \cone{\alpha}} \). Thus, \(\KK_0\) is a compact set -- contained in \(\Dz\) and not intercepting the stable manifold -- composed of two closed arcs of a circumference. In particular, there exists \(n_0 \in \NN\) such that:
		\begin{equation} \label{eq:n0}
			\text{if } n_x = \min \set{n \in \NN}{g^n (x) \notin \cl{\Dz}} \vg \, \text{ then } \, n_x \leq n_0 \fevery x \in \KK_0 \pt
		\end{equation}		
		Note that \(r_1 < \rhoz\). So, by the graph characterization of \(\stman\), for some \(\tau_1\) with \(\abs{\tau_1} \leq \min \roost{\tau, \betam}\) we have that \(r_1 \e{\ii \tau_1} \in \stman\), as suggested by \hyperref[fig:lipgraph]{Figure \ref{fig:lipgraph}}. 
		
		\begin{figure}[htb]
			\centering
			\includegraphics[scale=1.1]{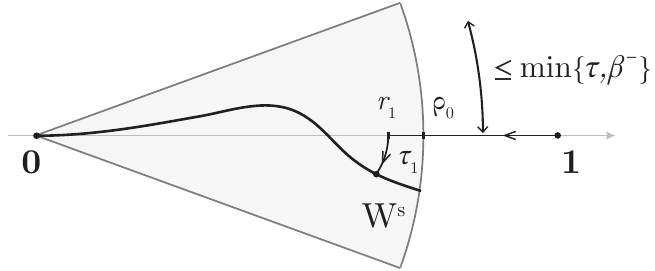}
			\caption{Since the stable manifold \(\stman\) is locally given as a Lipschitz graph \(y = y(x)\), it can be reached from the \(x\)-axis by an uniformly bounded rotation.}
			\label{fig:lipgraph}
		\end{figure}
		
		\begin{claim}
			Define \(M_1 (z) = r_1 \e{\ii \tau_1} z\). Then, the mapping \(M_1\) has the following properties:
			\begin{enumerate}[(i)]
				\item \(M_1 \in \mob[2]{\US}\),
				\item \(v_1 \eqdef M_1 \pts{\1} \in \stman\),
				\item \(w_1 \eqdef M_1 \pts{\zz} \in \KK_0\).
			\end{enumerate}			
		\end{claim}
		\begin{proof}[Proof of claim]
			Item (i) follows from the form of \(M_1\) and (ii) from the choice of \(\tau_1\). As for (iii), consider \(w_1 = M_1(z_0) = \e{\ii \tau_1} \, r_1 z_0\). Since \(r_1 z_0\) has the same polar angle \(\ttz\) as \(z_0\), \(\theta \pts{w_1} = \ttz + \tau_1\). But \(\abs{\tau_1} \leq \betam < \alpha\) and \(2 \alpha < \ttz < \pi - 2 \alpha\), which together imply \(\ttz + \abs{\tau_1} < \pi - \alpha\) and \(\ttz - \abs{\tau_1} > \alpha\). It follows that \(\alpha < \ttz + \tau_1 < \pi - \alpha\), 	allowing us to conclude that \(w_1 \notin \cone{\alpha}\). Furthermore, \(\abs{w_1} = \rho_1\), yielding \(w_1 \in \KK_0\).
		\end{proof}
		
		In particular, \(n_1 \eqdef n_{w_1}\) as in \eqref{eq:n0} is well-defined.
		
		\begin{claim}
			Consider \(f_t = g_{t - \floor{t}} \circ g^{\floor{t}} \circ M_1\), for \(0 < t \leq n_1\). Then, \(f\) is an isotopy satisfying \(f_t \pts{\zz} \notin \cone{\betap}\) and \(f_t \pts{\1} \in \cone{\betap}\), for every \(t \in \opclint{0,n_1}\).
		\end{claim}
		\begin{proof}[Proof of claim]
			On the one hand, \(f_t (z_0) = g_{t - \floor{t}} \, \pts{g^{\floor{t}} \pts{w_1}}\). Since \(w_1 \notin \cone{\alpha}\), each \(g^{\floor{t}} \pts{w_1}\) does not belong to \(\cone{\alpha}\) either, as observed in \hyperref[lemma:projsep]{Lemma \ref{lemma:projsep}}. However, it does belong to \(\Dz\) while \(t < n_1\). Thus, \(g_s \pts{g^{\floor{t}} \pts{w_1}} \notin \cone{\betap}\) for every \(s = t - \floor{t} \in \II\). On the other hand, \(f_t \pts{\1} = g_{t - \floor{t}} \, \pts{g^{\floor{t}} \pts{v_1}}\). Since \(v_1 \in \stman\), each \(g^{\floor{t}} \pts{v_1}\) belongs to \(\stman \cap \Dz \subset \cone{\betam} \cap \Dz\). Therefore, \(g_s \pts{g^{\floor{t}} \pts{v_1}} \in \cone{\betap}\) for every \(s = t - \floor{t} \in [0,1]\), as claimed.
		\end{proof}
		
		This setting is pictorially represented in \hyperref[fig:isotopy]{Figure \ref{fig:isotopy}}, where the points \(	z_1 \eqdef f_{n_1} \pts{\zz} = g^{n_1} \pts{w_1} \notin \cone{\alpha}\) and \(u_1 \eqdef f_{n_1} \pts{\1} = g^{n_1} \pts{v_1} \in \stman\) were introduced.
		
		\begin{figure}[htb]
			\includegraphics[scale=1.1]{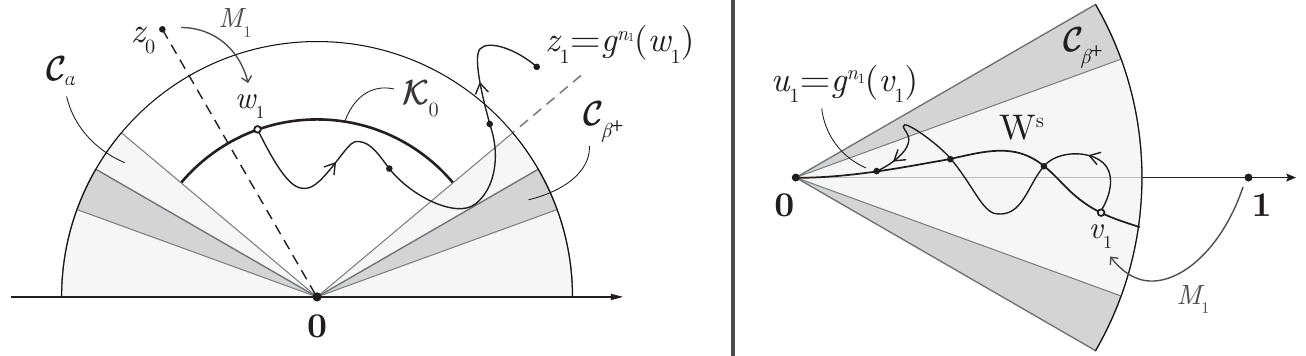}
			\caption{\(f_t\) promotes two parallel processes: points in \(\mathcal{K}_0\) are successively dragged out of \(\Dz\) without entering the \(\betap\)-cone, while the images of \(\1\) are dragged towards the origin over the stable manifold.}
			\label{fig:isotopy}
		\end{figure}
		
		By the choice of \(n_1\), \(|z_1| > \rhoz\). Let \(r_2 = \rho_1 / |z_1|\). Since product by \(r_2\) does not change angles, \(r_2 \, z_1 \in \KK_0\) and \(r_2 \, u_1 \in \Dz \cap \cone{\min \roost{\tau, \betam}}\). In particular, \(\e{\ii \tau_2} \pts{ r_2 \, u_1} \in \stman\) for some \(\tau_2\) with \(\abs{\tau_2} \leq \min \roost{\tau, \betam}\).
		
		\begin{claim}
			Define \(M_2 (z) = r_2 \e{\ii \tau_2} z\). Then, the mapping \(M_2\) has the following properties:
			\begin{enumerate}[(i)]
				\item \(M_2 \in \mob[2]{\US} \),
				\item \(v_2 \eqdef M_2 \pts{u_1} \in \stman\),
				\item \(w_2 \eqdef M_2 \pts{z_1} \in \KK_0\).
			\end{enumerate}	
		\end{claim}
		\begin{proof}[Proof of claim]
			Item (i) follows from the form of \(M_1\) and (ii) from the choice of \(\tau_2\). As for (iii), recall that \(r_2 z_1\) belongs to \(\Dz \setminus \cone{\alpha}\), which yields \(\alpha < \theta \pts{r_2 z_1} - \tau <  \theta \pts{r_2 z_1} + \tau < \pi - \alpha\). This implies, along with \(\abs{\tau_2} \leq \tau\), 	that \(w_2 \notin \cone{\alpha}\). Furthermore, \(\abs{w_2} = \abs{r_2 \, z_1} = \rho_1\), establishing that \(w_2 \in \KK_0\).
		\end{proof}
		
		In particular, \(n_2 \eqdef n_{w_2}\) is well-defined. Consider now the expression \(f_t = g_{t - \floor{t}} \circ g^{\floor{t} - n_1} \circ M_2 \circ f_{n_1}\), for \(n_1 < t \leq n_1 + n_2\). By arguments analogous to the ones developed previously, we have that \(t \mapsto f_t\) is continuous over the interval \(\opclint{n_1, n_1 + n_2}\). Also, \(f_t \pts{\zz} \notin \cone{\betap}\) and \(f_t \pts{\1} \in \cone{\betap}\), for every \(t \in \opclint{n_1, n_1 + n_2}\). In a similar fashion, we define inductively, for \(k \geq 0\):
		\begin{equation}
			f_t = \begin{cases}
				\id & \text{ if } t = 0 \vg \\
				g_{t - \floor{t}} \circ g^{\floor{t} - N_k} \circ M_{k+1} \circ f_{N_k} & \text{ on the interval } N_k < t \leq N_{k+1} \vg
			\end{cases} 
		\end{equation}
		where \(N_0 = 0\), \(N_k = \ds \sum_{i=1}^{k} n_i\) and the numbers \(n_k\) and the mappings \(M_k\) are determined as follows:
		\begin{itemize}
			\item \(M_{k+1} \in \mob[2]{\US}\) is a transformation of the form
			\[
				M_{k+1} \, (z) = r_{k+1} \, \e{\ii \tau_{k+1}} z \vg
			\]
			mapping \(z_k \eqdef f_{N_k} \pts{\zz} \notin \cl{\Dz}\) to a point \(w_{k+1} \in \KK_0\) and \(u_k \eqdef f_{N_k} \pts{\1} \in \stman\) to a point \(v_{k+1} \in \stman\), via an homothety of scaling factor \(r_{k+1} = \rho_1 / |z_k| < 1\) and a rotation of angle \(\abs{\tau_{k+1}} \leq \min \roost{\tau, \betam}\);
			\item \(	n_k \eqdef n_{w_k} \leq n_0\) is given as in \eqref{eq:n0}.
		\end{itemize}
		The following properties hold, by construction:
		\begin{enumerate}[(i)] 
			\item \(f_t \in G_2\) for every \(t \geq 0\),
			\item \(t \mapsto f_t\) is continuous over each interval of the form \(\opclint{N_k, N_{k+1}}\),
			\item \(f_t \pts{\zz} \notin \cone{\betap}\) and \(f_t \pts{\1} \in \cone{\betap}\) for every \(t \geq 0\).
		\end{enumerate}
		
		\begin{claim}
			For \(t \geq 0\), let
			\begin{equation}
				I_{t}^{\zz} = \mbar{ \, f_{t} \pts{\1} \, } \circ f_t \vg 
			\end{equation}	
			where \(\mbar{\cdot}\) is as in \hyperref[defi:mbar]{Definition \ref{defi:mbar}}. Then, \(\seq{I^{\zz}_{t}}{t \geq 0}\) is an \(\mathcal{I}G_3\)-isotopy.
		\end{claim}
		\begin{proof}[Proof of claim]
			Since \(f_t \in G_2\) for every \(t \geq 0\), it is clear that \(I^{\zz}_{t} \in G_3\) for every \(t \geq 0\). It is left to verify that \(t \mapsto I^{\zz}_{t}\) defines a continuous curve of homeomorphisms. This mapping is, \emph{a priori}, as continuous as \(t \mapsto f_t\). Thus, all that is needed to check is continuity from the right at the left endpoints of each interval \(\opclint{N_k, N_{k+1}}\). For \(0 < h < 1\),
			\[
				I^{\zz}_{N_k + h} = \mbar{ \, f_{N_k + h} \pts{\1} \, } \circ f_{N_k + h} = \mbar{ \, g_{h} \pts{ v_{k+1} }  \, } \circ g_{h} \circ M_{k+1} \circ f_{N_k} \pt
			\]
			But notice that
			\[
				\mbar{u_k} \circ M_{k+1}^{-1} \circ \mxy{ \, g_h \pts{v_{k+1}} }{ v_{k+1} \,} 
			\]
			is \emph{a} Möbius transformation fixing the poles and mapping \(g_h \pts{v_{k+1}}\) to \(\1\). By sharp 3-transitivity, it must be \emph{the} transformation \(\mbar{ g_{h} \pts{ v_{k+1} } }\). But, since \(g_{h} \to \id\) as \(h \to 0^{+}\), the above expression implies \(\mbar{ g_{h} \pts{ v_{k+1} } } \to \mbar{ u_{k} } \circ M_{k+1}^{-1}\) as \(h \to 0^{+}\). Consequently,
			\begin{equation*}
				I^{\zz}_{N_k + h} \to \mbar{ u_k } \circ M_{k+1}^{-1} \circ g_{0} \circ M_{k+1} \circ \, f_{N_k} = \mbar{ u_k } \circ f_{N_k} = \mbar{ \, f_{N_k} \pts{\1} \, } \circ f_{N_k} = I^{\zz}_{N_k} \pt \qedhere
			\end{equation*}
		\end{proof}
		
		\begin{claim}
			\(\orb[I^{\zz}]{\zz} \cap \Gamma = \emptyset
			 \pt\)
		\end{claim}
		\begin{proof}[Proof of claim]
			Since \(f_t \pts{\1} \in \cone{\betap}\) for every \(t \geq 0\), the transformation \(\mbar{ \, f_t \pts{\1} \, }\) may be explicitly written as \(\mbar{ \, f_t \pts{ \1 } \, } (z) = |f_t \pts{ \1 } |^{-1} \, \e{ \ii \psi} z\), where \(\abs{\psi} \leq \betap\). We also known that \(f_{t} \pts{\zz} \notin \cone{\betap}\) for every \(t \geq 0\). Thus, \(f_t \pts{\zz} = \abs{f_t \pts{\zz}} \, \e{\ii \theta}\), where \(\betap < \theta < \pi - \betap\). Consequently, \(0 < \theta + \psi < \pi\). But \(\theta \pts{I^{\zz}_{t} (\zz) } = \theta + \psi\). Therefore, the trajectory \(\orb[I^{\zz}]{\zz}\) remains on the upper half-plane without ever touching the \(x\)-axis. This establishes \textbf{(1)}.
		\end{proof}
		
		\begin{claim}
			\(f_t \pts{\1} \to \0\) as \(t \to + \infty\).
		\end{claim}
		\begin{proof}[Proof of claim]
			We may assume that \(\abs{g(z)} < \abs{z}\) for every \(z \in \stman\). Then, since \(\abs{u_{k}} = | g^{n_{k}} \pts{v_{k}}| < \abs{v_{k}}\) and \(\abs{v_{k+1}} = r_{k+1} \abs{u_k}\), we have 
			\[
				\abs{v_{k+1}} \leq \pts{{\rho_1}/{\rho_0}}^{k+1} \pt
			\]	
			But, on each interval \(\opclint{N_k, N_{k+1}}\), we have \(f_t \pts{\1} = g_{t - \floor{t}} \circ g^{\floor{t} - N_{k}} \pts{ v_{k+1} } \). As \(t\) ranges through this interval, the quantity \(t - \floor{t}\) ranges over the interval \([0,\!1]\) and the quantity \(\floor{t} - N_k\) ranges through \(\roost{0, \ldots, n_{k+1}} \subset \roost{0, \ldots, n_0}\). Thus,
			\begin{equation} \label{eq:sandwich}		
					\sup \set{ \abs{ f_t \pts{\1} } }{ t > N_k} \leq \max \set{ \abs{g^{i}_{s} \pts{z}} }{ s \in \II \, , \, 1 \leq i \leq n_0 \aand z \in \cdisk{\0}{\pts{\rho_1 / \rhoz }^{k+1}} } \pt
			\end{equation}				
			Let \(\eps > 0\) be given. Since each \(\seq{g^{i}_{t}}{t \in \II}\) is an isotopy fixing the origin, we may find \(\eta > 0\) such that  \(|{g^i_{s} \pts{z}}| < \eps\) for every \(s \in \II\) and \(1 \leq i \leq n_0\), whenever \(\abs{z} < \eta \). Consequently, if \(k_0 \in \NN\) is so large that \(\pts{\rho_1 / \rhoz}^{k_0+1} < \eta\), \eqref{eq:sandwich} implies \( \abs{f_t \pts{\1}} < \eps\) whenever \(t > N_{k_0}\).
		\end{proof}
		
		We are, now, ready to finish the proof. We know that \(f_t \pts{\zz} = g_{t - \floor{t}} \circ g^{\floor{t} - N_{k}} \pts{ w_{k+1} } \) on each interval \(\opclint{N_k, N_{k+1}}\), where \(w_k \in \KK_0\) for every \(k \in \NN\). By observing the same ranges as in the proof of the previous claim, we see that, for every \(t \geq 0\),
			\begin{equation*}
				f_t \pts{z_0} \in \KK \eqdef \set{ g^{i}_{s} \pts{z} }{ s \in \II \, , \, 1 \leq i \leq n_0 \, \aand \, z \in \KK_0 } \pt
			\end{equation*}
			But \(\KK\) is a compact set bounded away from \(\0\). Since \(f_t \pts{\1} \to \0\) as \(t \to + \infty\) and \(I^{\zz} (\zz) \in \mbar{ \, f_t \pts{\1} \, } \pts{\KK}\), the remark in \eqref{eq:auxlim} implies \textbf{(2)}.
	\end{proof}
	
	\begin{cor} \label{cor:alfaomega}
		Let \(G \subset \diff{1}{\US}\) be a proper extension of \(\mob{\US}\). Then, for each \(\zz\) not on the meridian \(\Gamma\) and each pair of distinct points \(a, b \in \roost{\0,\1,\inff}\) there exists a full-time \(\mathcal{I}G_3\)-isotopy \(\seq{\, I^{\zz}_{ab} \pts{t , \; \cdot } \,}{t \in \RR}\) such that:
		\begin{enumerate}[\bfseries (1)]
			\item the trajectory of \(\zz\) under \(I^{\zz}_{ab}\) does not intercept \(\Gamma\),			
			\item the \(\alpha\) and \(\omega\) limits of \(\zz\) satisfy \(\alphalim{ \zz }{ I^{\zz}_{ab} } = \roost{a}\) and \(\omegalim{ \zz }{ I^{\zz}_{ab} } = \roost{b}\).
		\end{enumerate}
	\end{cor}	
	\begin{proof}
		Let \(\zz \notin \Gamma\) and \(a, b \in \roost{\0,\1,\inff}\) be given. We denote by \(\T{a}{b}\) the (unique) idempotent Möbius transformation permuting \(a\) and \(b\) and fixing the remaining reference point. Each \(\T{a}{b}\) leaves \(\Gamma\) invariant, so neither \(\T{\inff}{b} \pts{\zz}\) nor \(\T{\inff}{a} \pts{\zz}\) lie on \(\Gamma\). The \hyperref[lemma:fundamentallemma]{Fundamental Lemma} then yields \(\mathcal{I}G_3\)-isotopies \( ( \, I^{ \T{\inff}{b} (\zz) \, }_{t})_{t \geq 0} \) and \( ( \, I^{ \T{\inff}{a} (\zz) }_{t} \, )_{t \geq 0}\) as described therein, from which we define \(I^{\zz}_{ab} : \RR \times \US \to \US\) as:
		\begin{equation}
			I^{\zz}_{ab} \pts{t, \, z} = \begin{cases} 
				\T{\inff}{b} \circ I^{ \T{\inff}{b} \pts{\zz} }_{t} \circ \T{\inff}{b} \; (z) & \text{if } t \geq 0 \vg \\[0.25em]
				\T{\inff}{a} \circ I^{ \T{\inff}{a} \pts{\zz} }_{-t} \circ \T{\inff}{a} \; (z) & \text{if } t \leq 0 \pt
			\end{cases} 
		\end{equation}
		Each mapping \(\restric{ I^{\zz}_{ab} }{ \clopint{0,+\infty} \times \US }\) and \(\restric{ I^{\zz}_{ab} }{ \opclint{-\infty,0} \times \US }\) is itself an isotopy, and they agree in the common slice \(\roost{0} \times \US\), both being equal to \(\id\) there. Consequently, \(I^{\zz}_{ab}\) defines a global jointly continuous function of the variables \(\pts{t,z} \in \RR \times \US\). The statements then follows from the \hyperref[lemma:fundamentallemma]{Fundamental Lemma}, noticing also that reversion of time turns the \(\omega\)-limit into the \(\alpha\)-limit.
	\end{proof}
	
\section{A Crossing Lemma} \label{sec:cross}

Given \(r>0\), we let \(\gamma_r : \II \to \RR\) be the path \(\gamma_r \pts{\theta} = r \e{2\pi \ii \theta}\) and \(S_r\) be its image. Then, if \(g \in \diff{1}{\RR^2}\) is a planar diffeomorphism fixing the origin, we can check by a calculation that:
		\begin{equation} \label{eq:c1conv}
			\frac{g \circ \gamma_r}{r} \xrightarrow{ \reg{1} \pts{\II} } \dif{g}{\0} \circ \gamma_1 \quad \aas \quad r \to 0^+ \pt
		\end{equation}
		
		Assume further that \(\dif{g}{\0}\) is a hyperbolic saddle matrix. Then, it maps \(S_1\) onto an ellipse having semiminor axis of length strictly smaller than one and semimajor axis of length strictly greater than one. Thus, \(S_1 \cap \dif{g}{\0} \pts{S_1}\) consists of four nontangential intersections. It then follows from \eqref{eq:c1conv}, along with the stability of nontangential intersections, that for a sufficiently small \(r_0 > 0\) (indeed, for infinitely many), \(S_{r_0} \cap g \pts{ S_{r_0} }\) contains exactly four points.
	
	This remark allows us to now analyze the behaviour of \(\Gamma\) under the action of \(G\). Before doing so, we agree that given three distinct points \(a,b,c\), we denote by \(\mabc{a}{b}{c}\) the unique Möbius transformation mapping \(a\) to \(\0\), \(b\) to \(\1\) and \(c\) to \(\inff\). If the points are all finite and nonzero, this transformation can be explicitly written as a certain cross ratio, from which it is seen that the association \(\pts{a,b,c} \mapsto \mabc{a}{b}{c}\) is continuous.	
	
	\begin{lemma} \label{lemma:4pt}
		Let \(G \subset \diff{1}{\US}\) be a proper extension of \(\mob{\US}\). Then, there exists an \(\mathcal{I}G_3\)-isotopy \(\seq{k_t}{t \in \II}\) such that \(\abs{\, k_1 \pts{\Gamma} \cap \Gamma \,} = 4\).
	\end{lemma}
	\begin{proof}
		By the \hyperref[lemma:extlemma]{Extension Lemma}, it is known that there exists an \(\mathcal{I}G_2\)-isotopy \(\seq{g_t}{t \in \II}\) such that \(g \eqdef g_1\) has a hyperbolic saddle fixed point at the origin. So, we may fix any sufficiently small \(0<r_0<1\) such that \(\abs{\, S_{r_0} \cap g \pts{ S_{r_0} } \,} = 4\).
		
		Let \(a,b,c,d \in S_{r_0}\) be four consecutive points in the usual anticlockwise cyclic order, and such that \(S_{r_0} \cap g \pts{ S_{r_0} } = \roost{g(a), g(b), g(c), g(d)}\). Then, we may consider the unique \(M_0 \in \mob{\US}\) mapping the ordered triple \(\pts{\0,\1,\inff}\) onto \(\pts{a,b,c}\). Notice that \(M_0 = \mabc{a}{b}{c}^{-1}\) and that \(M_0 \pts{\Gamma} = S_{r_0}\). By defining
		\begin{equation*}
			k_t = \mabc{g_t(a)}{g_t(b)}{g_t(c)} \circ g_t \circ M_0 \; \vg \quad 0 \leq t \leq 1 \; \vg
		\end{equation*}
		we obtain an \(\mathcal{I}G_3\)-isotopy \(\seq{k_t}{t \in \II}\). Since \(M_0\) is a bijection of the sphere:
		\begin{align*}
			\abs{\, k_1 \pts{\Gamma} \cap \Gamma \,} &= \abs{\,  \pts{M_0 \circ \mabc{g(a)}{g(b)}{g(c)}} \circ g \pts{ M_0 \pts{\Gamma}} \cap M_0 \pts{\Gamma} \,} \\
			&=  \abs{\,  g \pts{S_{r_0}} \cap S_{r_0} \,} = 4 \pt \qedhere
		\end{align*}
	\end{proof}
	
	Finite points \(a,b\) on the meridian \(\Gamma\) are identified with their real counterparts on the \(x\)-axis. This induces a natural ordering, for which one may speak of the oriented interval with endpoints \(a,b\). Whenever \(a \leq b\) with respect to this ordering,
	\begin{itemize}
		\item \(\clint{a,\!b}\) denotes the arc of the meridian \(\Gamma\) with endpoints \(a,b\) and not containing \(\inff\), which corresponds to the compact interval of the \(x\)-axis with the associated endpoints;
		\item \(\clint{b,\!a}\) denotes the arc of the meridian \(\Gamma\) with endpoints \(a,b\) and containing \(\inff\), which projects onto \(\opclint{-\infty,\!a} \cup \clopint{b, \! + \infty}\).
	\end{itemize}
	If \(b = \inff\), then the corresponding arcs are defined via stereographic projection as \(\clint{\inff,\! a} = \opclint{- \infty,\! a} \cup \roost{\inff}\) and \(\clint{a, \! \inff} = \clopint{a, \! + \infty} \cup \roost{\inff}\). Lastly, open and half-open arcs of \(\Gamma\) are defined accordingly by deletion of the suitable endpoints from the corresponding closed arcs.	
	
	\begin{cor} \label{cor:outofgamma}
		Let \(G \subset \diff{1}{\US}\) be a proper extension of \(\mob{\US}\). Then, given \(z_0 \in \Gamma \setminus \roost{\0,\1,\inff}\), there exists an \(\mathcal{I}G_3\)-isotopy \(\seq{h_t}{t \in \II}\) such that \(h_1 \pts{z_0} \notin \Gamma\).
	\end{cor}
	\begin{proof}
		In the language of \hyperref[lemma:4pt]{Lemma \ref{lemma:4pt}}, let \(k_1 \pts{\Gamma} \cap \Gamma = \roost{\0, \1, \inff, \wz}\).  If \(k_1 (\zz) \neq \wz\), it suffices to take \(h_t = k_t\) for every \(t\). Otherwise, assume for concreteness that \(\zz \in \opint{\0,\!\1}\) -- analogous reasonings apply to the other cases. By defining \(	h_t = \mabc{\1}{\inff}{\0} \circ k_t \circ \mabc{\1}{\inff}{\0}^{-1}\), we obtain a new \(\mathcal{I}G_3\)-isotopy. Since \( \restric{\mabc{\1}{\inff}{\0}}{\Gamma}\) is an interval exchange transformation free of fixed points, \(\zzt \eqdef \mabc{\1}{\inff}{\0}^{-1} (\zz) \neq \zz\). Thus, \(k_1 \pts{\zzt} \notin \Gamma\). But \(\mabc{\1}{\inff}{\0}\) leaves \(\Gamma\) invariant, so \(h_1 \pts{\zz} = \mabc{\1}{\inff}{\0} \pts{ k_1 (\zzt) } \notin \Gamma\) as well.
	\end{proof}
	
	\begin{lemma}[The Crossing Lemma] \label{lemma:crlemma}
		Let \(G \subset \diff{1}{\US}\) be a proper extension of \(\mob{\US}\). Then, there exist a point \(\zh\) on the open segment \(\opint{\0,\!\1}\) of \(\Gamma\) and an \(\mathcal{I}G_3\)-isotopy \(\seq{J_t}{t \in \clint{-1,1}}\) such that:
		\begin{enumerate}[\bfseries (1)]
			\item the trajectory of \(\zh\) under \(J\) only intercepts \(\Gamma\) on the arc \(\opint{\0,\! \1}\),
			\item \(J_{-1} \pts{\zh} \in \hm\) and \(\, J_{\, 1} \pts{\zh} \in \hp\).
		\end{enumerate}
	\end{lemma}
	\begin{proof}
		By \hyperref[cor:outofgamma]{Corollary \ref{cor:outofgamma}}, we may fix an \(\mathcal{I}G_3\)-isotopy \(\pts{h_t}_{t \in \II}\) and a point \(\zz \in \Gamma\) such that \(h_1 \pts{ \zz } \notin \Gamma\). Without loss of generality, assume \(h_1 \pts{\zz} \in \hp\). By continuity, there is some \(\uz \in \hm\) near \(\zz\) such that \(h_1 \pts{ \uz } \in \hp\) as well. Consider the continuous path \(\gamma \colon \II \to \	\US\) defined as \(\gamma (t) = h_t \pts{\uz}\). The sets \(\gamma^{-1} \pts{\hm}\), \(\gamma^{-1} \pts{\Gamma}\) and \(\gamma^{-1} \pts{\hp}\) form a partition of \(\II\). Let
		\[
			\tm = \sup \gamma^{-1} \pts{\hm} \quad \aand  \quad \tp = \inf \gamma^{-1} \pts{\hp} \cap \clint{\tm, \! 1} \pt
		\]
		We have that \(0 < \tm \leq \tp < 1\), and that \(\clint{\tm, \tp} \subset \gamma^{-1} \pts{\Gamma}\). Since \(\gamma\) cannot intercept the points \(\roost{\0,\1,\inff}\), it follows that \(\gamma \pts{ \clint{\tm, \tp} }\) is a compact arc contained within one of the three connected components of \(\Gamma \setminus \roost{\0, \1, \inff}\). It is, thus, at a positive distance \(\rho > 0\) from those three points. Let \(\delta > 0\) be such that 
		\begin{equation} \label{eq:deltapm}
			\begin{gathered}
				s \in \II \, \aand \, \abs{s - \tm} < \delta \; \text{imply} \; d \pts{\gamma (s), \gamma (\tm)} < \rho \,  \text{;} \\
				s \in \II \, \aand \, \abs{s - \tp} < \delta \; \text{imply} \; d \pts{\gamma (s), \gamma (\tp)} < \rho \pt		
			\end{gathered}
		\end{equation}		
		
		\begin{figure}[htb]
			\includegraphics[scale=1.1]{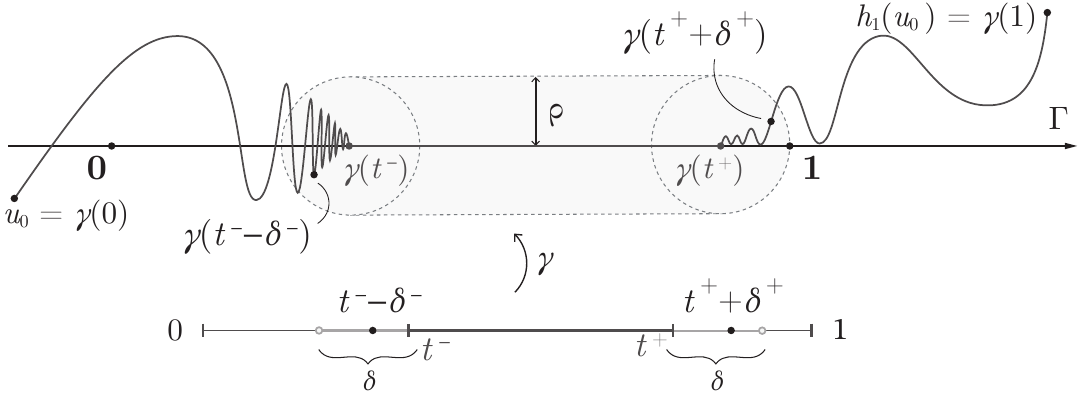}
			\caption{After leaving the eastern hemisphere, the isotopy path of \(\uz\) under \(h\) may remain trapped on a compact segment of \(\Gamma \setminus \roost{\0,\1,\inff}\) before entering the western hemisphere.}
			\label{fig:crossing}
		\end{figure}
		
		Then, we may find points of the form \(\tm - \dtm \in \gamma^{-1} \pts{\hm}\) and \(\tp + \dtp \in \gamma^{-1} \pts{\hp}\) satisfying \(\tm - \delta < \tm - \dtm < \tm \leq \tp < \tp + \dtp < \tp + \delta\) -- as conveyed in \hyperref[fig:crossing]{Figure \ref{fig:crossing}} -- and let \(\sigma: \clint{-1,1} \to \clint{\tm - \dtm, \tp + \dtp}\) be any increasing bijection such that \(\sigma (0)\) is the midpoint of the interval \(\clint{\tm, \tp}\). 
		
		Defining \(\tilde{J}: \clint{-1,1} \times \US \to \US\) as \(\tilde{J} \pts{t,z} = \tilde{J}_t (z) = h_{\sigma (t)} \circ h_{\sigma (0)}^{-1} \, (z)\), 	we readily see that it is an \(\mathcal{I}G_3\)-isotopy. By letting \(\zt = \gamma \pts{ \sigma (0) } = h_{\sigma (0) } \pts{\uz} \in \Gamma\), we have that \(\tilde{J}_{\, 1} \pts{ \zt } = \gamma (\tp + \dtp) \in \hp\) and \(\tilde{J}_{-1} \pts{ \zt } = \gamma (\tm - \dtm) \in \hm\). Also, by observing the range of \(\sigma\), the trajectory of \(\zt\) under \(\tilde{J}\) is \(\orb[\tilde{J}]{\zt} = \gamma \pts{ \clint{ \tm - \dtm , \tp + \dtp } }\). In particular, it follows from the choice of \(\rho\) and from \eqref{eq:deltapm} that any point in \(\orb[\tilde{J}]{\zt} \cap \Gamma\) lies on the same connected component of \(\Gamma \setminus \roost{\0, \1, \inff}\) as the segment \(\gamma \pts{ \clint{\tm, \tp} }\), which is precisely the component containing \(\zt\).	
		
		If this component happens to be \(\opint{\0,\!\1}\), as pictured in \hyperref[fig:crossing]{Figure \ref{fig:crossing}}, we simply let \(J = \tilde{J}\) and \(\zh = \zt\). Otherwise, consider \(\Gamma\) endowed with the cyclic order induced by the real line. If \(\zt \in \opint{a,\! b}\), let \(\roost{c} = \roost{\0,\1,\inff} \setminus \roost{a,\! b}\). Since \(\mabc{a}{b}{c}\) applies \(\opint{a,\! b}\) onto \(\opint{\0,\! \1}\) in an orientation-preserving way, one may take \(\zh = \mabc{a}{b}{c} \pts{\zt}\) and \(J = \mabc{a}{b}{c} \circ \tilde{J}_{\, t} \circ \mabc{a}{b}{c}^{-1}\) to obtain the desired point and isotopy.
	\end{proof}
	
\section{Conclusions} \label{sec:closing}	

\subsection{Proof of Theorem A}

	From now on, a proper extension \(G \subset \diff{1}{\US}\) of \(\mob{\US}\) is fixed throughout, and we let the point \(\zh \in \Gamma\) and the \(\mathcal{I}G_3\)-isotopy \(\seq{J_t}{t \in [-1,1]}\) be as in the \hyperref[lemma:crlemma]{Crossing Lemma}. Upon denoting \(\zhm = J_{-1} \pts{\zh}\) and \(\zhp = J_{1} \pts{\zh}\), we define the set
	\begin{equation} \label{eq:chi} 
		\chi = \cl{ \orb[I^{\zhm}]{\zhm} \cup \orb[J]{\zh} \cup \orb[I^{\zhp}]{\zhp} } \pt 
	\end{equation}	
	Above, \(I^{\zhm}\) and \(I^{\zhp}\) are the isotopies yielded by the \hyperref[lemma:fundamentallemma]{Fundamental Lemma} when considering the points \(\zhm \in \hm\) and \(\zhp \in \hp\). On the sphere, \(\chi\) is a continuum. Indeed, 
	\(
		\orb[I^{\zhm}]{\zhm} \cup \orb[J]{\zh} \cup \orb[I^{\zhp}]{\zhp}
	\)
	is connected, as it is the union of (connected) curves with points in common, while its closure is automatically compact on the compact space \(\US\), and consists of adjoining \(\roost{\inff}\) to this union, as we suggest in \hyperref[fig:continuum]{Figure \ref{fig:continuum}}.
		
	\begin{figure}[htb]
		\centering
		\includegraphics[scale=1.1]{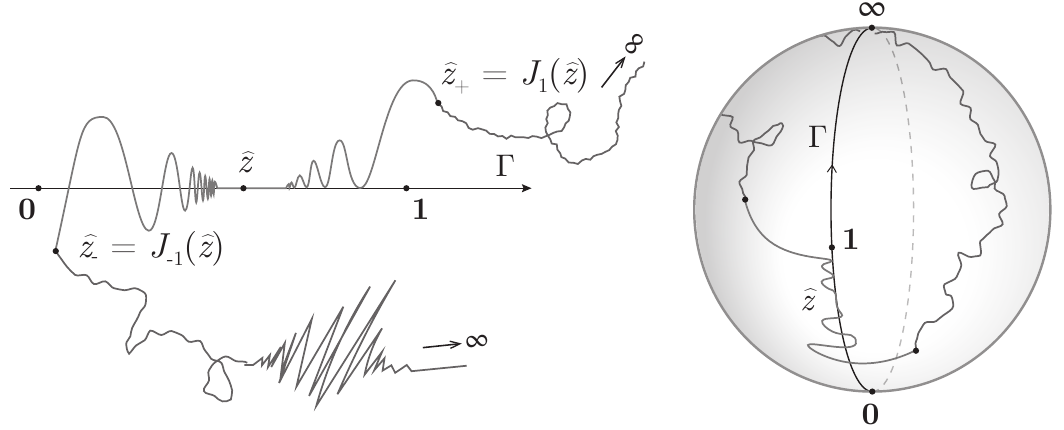}
		\caption{The continuum \(\chi\) is constructed by gluing together isotopy trajectories with points in common, one of which is bounded away from \(\inff\) and two of which are known to accumulate at \(\inff\), and then taking their closure.}
		\label{fig:continuum}
	\end{figure}
		
	\begin{auxclaim} \label{lemma:continuum}
		The set \(\chi\) is a continuum separating \(\opint{\inff,\! \0}\) and \(\opint{\1, \! \inff}\) in the following sense: whenever \(\alpha : \II \to \US\) is a path such that \(\alpha(0) \in \opint{\inff,\! \0}\) and \(\alpha(1) \in \opint{\1, \inff}\), we have that \(\alpha \pts{\II} \cap \chi \neq \emptyset\).
	\end{auxclaim}
	\begin{proof}
		If \(\inff \in \alpha \pts{\II}\) there is nothing to prove. Otherwise, suppose that \(\alpha\) never passes through \(\infty\). By a standard Zorn Lemma argument, it can also be assumed simple. We let
		\[
			\tm = \sup \alpha^{-1} \opclint{\inff,\! \0 }  \quad \aand \quad \tp = \inf \pts{ \alpha^{-1} \clopint{\1, \! \inff} } \cap \clint{\tm,\! 1} \pt
		\]
		Then, \(0 \leq \tm < \tp \leq 1\), and any intersection between \(\alpha (t)\) and \(\Gamma\) for \(\tm < t < \tp\) happens on the open segment \(\opint{\0,\! \1}\). 
		
		On the plane, we define a continuous mapping \(\ell : \RR \to \RR^2\) as pictured in \hyperref[fig:line]{Figure \ref{fig:line}}, and explicitly given by:
		\begin{equation*}
			\ell (t) = \begin{cases}
				\ds \frac{1-t \hphantom{^{-}} }{1-\tm} \pts{ \alpha (\tm) - \zh } + \zh & \: \text{if } t \leq \tm \vg \\[0.75em]
				\alpha (t) & \: \text{if } \tm < t < \tp \vg \\[0.75em]
				\ds \frac{1+t \hphantom{^{+}} }{1+\tp} \pts{ \alpha (\tp) - \zh} +\zh &\: \text{if } t \geq \tp \pt
			\end{cases} 
		\end{equation*}
		The mapping \(\ell\) is a line or, in other words, a simple and proper path. The orientation it inherits from the real line automatically divides the plane into  two open and connected components, the \emph{right} \(\rgt{\ell}\) and the \emph{left} \(\lft{\ell}\) of \(\ell\), plus their common boundary \(\ell\).
		
		\begin{figure}[htb]
			\includegraphics[scale=1.1]{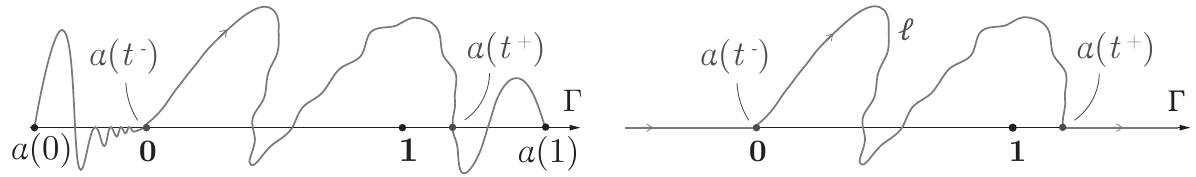}
			\caption{The segment of \(\alpha\) comprehended between the instant it leaves \(\opclint{\inff,\0}\) and the instant it enters \(\clopint{\1, \inff}\) can be glued to \(\Gamma\) -- traversed the usual way -- to generate a line \(\ell\).}
			\label{fig:line}
		\end{figure}		
				
		We now consider the compact set \(\alpha \pts{ \clint{\tm, \tp} }\), and fix some closed disk \(D \subset \RR^2\) fully containing it. Then, \(\clint{\0,\!\1} \subset D\). Thus, if we now consider the open set \(\mathcal{O} = \RR^2 \setminus D\), we see from the expression of \(\ell\) that \(\ell \cap \mathcal{O} = \Gamma \cap \mathcal{O}\). Also, \(\ell\) traverses this intersection with the same orientation as \(\Gamma\). It follows that \(\lft{\ell} \cap \mathcal{O} = \lft{\Gamma} \cap \mathcal{O} = \hp \cap \mathcal{O}\) and \(\rgt{\ell} \cap \mathcal{O} = \rgt{\Gamma} \cap  \mathcal{O} = \hm \cap \mathcal{O}\).
		
		From the \hyperref[lemma:fundamentallemma]{Fundamental Lemma}, \(\orb[ I^{\zhp} ]{\zhp}\) is fully contained within \(\hp\) and accumulates at \(\roost{\inff}\). Since \(\mathcal{O}\) is a neighbourhood of \(\inff\) on the sphere, it follows that \(\orb[ I^{\zhp} ]{\zhp} \cap \lft{\ell} \neq \emptyset\). Analogously, \(\orb[ I^{\zhm} ]{\zhm} \cap \rgt{\ell} \neq \emptyset\). This translates to \(\chi \cap \lft{\ell} \neq \emptyset\) and \(\chi \cap \rgt{\ell} \neq \emptyset\). Therefore, the line \(\ell\) intercepts the continuum \(\chi\).
		
		Let \(\tbr \in \RR\) be such that \(\ell (\tbr) \in \chi\). Then, it must be the case that \( \tbr \in \opint{\tm, \tp} \). Indeed, on the one hand, \(\ell \pts{\tbr} \in \clopint{\1, \! \inff}\) if \(\tbr \geq \tp\) and \(\ell \pts{\tbr} \in \opclint{\inff, \! \0}\) if \(\tbr \leq \tm\). On the other hand, any intersection between \(\chi\) and \(\Gamma\) must take place on the open segment \(\opint{\0, \! \1}\), by the \hyperref[lemma:crlemma]{Crossing Lemma} and the \hyperref[lemma:fundamentallemma]{Fundamental Lemma}. However, \(\tm < \tbr < \tp\) means that \(\ell \pts{\tbr} = \alpha \pts{\tbr}\), yielding an intersection between \(\alpha \pts{\II}\) and \(\chi\), as claimed.
	\end{proof}
	
	Next, consider the equivalence relation \(\gtequiv\), as described in \hyperref[defi:gequiv]{Definition \ref{defi:gequiv}}. Clearly, \(\acc{\0} = \roost{\0}\), \(\acc{\1} = \roost{\1}\) and \(\acc{\inff} = \roost{\inff}\). Our goal is to show that \(\zz \in \acc{\zh}\) for any \(\zz \in \US \setminus \roost{\0,\1,\inff}\). By \hyperref[cor:outofgamma]{Corollary \ref{cor:outofgamma}}, it suffices to consider \(\zz \notin \Gamma\).
	
	\begin{auxclaim} \label{lemma:acc}
		If \(\zz \notin \Gamma\), then \(\zz \in \acc{\zh}\).
	\end{auxclaim}
	\begin{proof}
		Let \(\chi\) be as in \eqref{eq:chi}. We fix \(r>0\) such that \(\cdisk{\0}{r} \cap \chi =\cdisk{\1}{r} \cap \chi = \cdisk{\0}{r} \cap \cdisk{\1}{r} = \emptyset\), and also such that both closed disks are disjoint from \(\roost{\zz}\) and \(\roost{\inff}\).	 Consider the \(\mathcal{I}G_3\)-isotopy \(I^{\zz}_{\0 \1}\) yielded by \hyperref[cor:alfaomega]{Corollary \ref{cor:alfaomega}}. We encounter \(S < 0\) \emph{maximal} such that \(I^{\zz}_{\0\1} (S, \zz) \in \fr{ \cdisk{\0}{r} }\) and \(T>0\) \emph{minimal} such that \(I^{\zz}_{\0\1} (T, \zz) \in \fr{ \cdisk{\1}{r} }\). Let \(\alpha: \II \to \RR^2\) be given by:
	\begin{equation} \label{eq:interpol}
		\alpha \, (t) = \begin{cases} 
			4t \; I^{\zz}_{\0\1} \pts{S, \, \zz} - \pts{1 - 4t} \, \ds \frac{r}{2} & \, \text{if } \ds 0 \leq t \leq \frac{1}{4} \vg \\[0.75em]
			I^{\zz}_{\0\1} \pts{ \, 2t \pts{T-S} + \ds \frac{3S - T}{2} \, , \, \zz} & \, \text{if } \ds \frac{1}{4} \leq t \leq \frac{3}{4} \vg \\[0.75em]
			\pts{4 - 4t} \; I^{\zz}_{\0\1} \pts{T, \, \zz} + \pts{4t - 3} \, \pts{ \1 + \ds \frac{r}{2} } & \, \text{if } \ds \frac{3}{4} \leq t \leq 1 \pt
		\end{cases}
	\end{equation}	 	
	Geometrically, \(\alpha\) departs from a point \(\alpha (0) \in \opint{\inff,\!\0} \cap \cdisk{\0}{r}\), and follows on a straight line until it reaches a certain point of \(\zz\)'s trajectory in the disk's boundary. This intersection point is such that \(\zz\)'s trajectory returns to this first disk at most finitely many times in the future. From it, the path \(\alpha\) follows \(\zz\)'s path until it first reaches the boundary of the disk \(\cdisk{\1}{r}\). Then, \(\alpha\) moves on a straight line until it reaches a point \(\alpha(1) \in \opint{\1, \!\inff} \cap \cdisk{\1}{r}\). This process is conveyed in \hyperref[fig:thma]{Figure \ref{fig:thma}}.
	
		\begin{figure}[htb]
			\centering
			\includegraphics[scale=1.1]{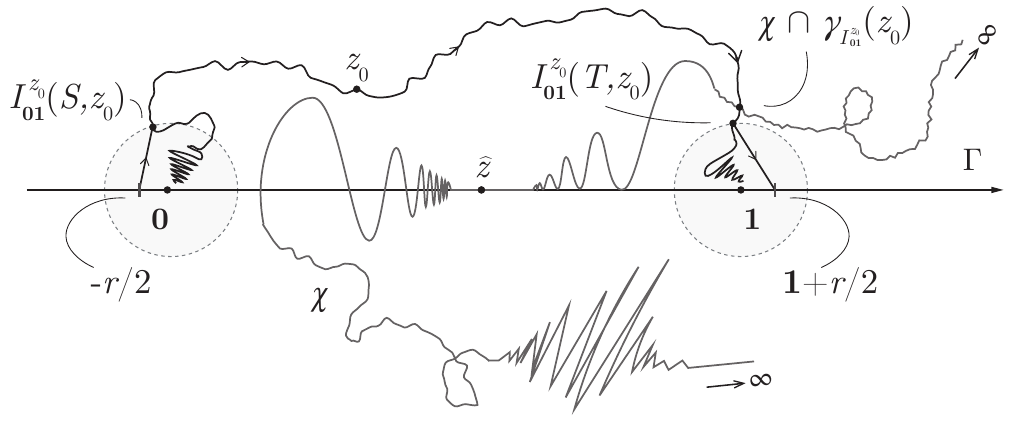}
			\caption{After leaving a compact neighbourhood of \(\0\) disjoint of \(\chi\) and before entering a neighbourhood of \(\1\) disjoint from \(\chi\), the path of \(\zz\) under \(I^{\zz}_{\0\1}\) must cross the continuum \(\chi\).}
			\label{fig:thma}
		\end{figure}
		
		In particular, \(\alpha (0) \in \opint{\inff,\!\0}\), \(\alpha(1) \in \opint{\1, \!\inff}\) and \(\alpha \pts{\II} \subset \US \setminus \roost{\inff}\). By \hyperref[lemma:continuum]{Claim \ref{lemma:continuum}}, we have that \(\alpha\) intercepts \(\chi \setminus \roost{\inff}\). But, since the segments \(\alpha \pts{ \clint{0,\!{1}/{4}}}\) and \(\alpha \pts{ \clint{{3}/{4},\! 1} }\) are contained within disks disjoint from \(\chi\), we must have \( \alpha \pts{ \, \opint{{1}/{4},\! {3}/{4}} \, } \cap \pts{ \chi \setminus \roost{\inff} } \neq \emptyset \). However, it is seen from \eqref{eq:interpol} that \(\alpha \pts{ \, \opint{{1}/{4},\! {3}/{4}} \, } \subset \orb[ I^{\zz}_{\0\1} ]{\zz}\). By \hyperref[lemma:equivcross]{Lemma \ref{lemma:equivcross}}, this is readily seen to imply \(\zz \in \acc{\zh}\).		
	\end{proof}
	
	This is enough to derive the (arc) 4-transitivity of \(G_0\) for, if \(\pts{a,b,c,d}\) and \(\pts{p,q,r,s}\) are two given lists of distinct points on the sphere, let \(\zz = \mabc{a}{b}{c} \pts{d}\) and \(\wz = \mabc{p}{q}{r}^{-1} (s)\). Then, neither \(\zz\) nor \(\wz\) belong to \(\roost{\0,\1,\inff}\) and thus, by \hyperref[lemma:acc]{Claim \ref{lemma:acc}}, both \(\zz\) and \(\wz\) belong to \(\acc{\zh}\). This implies that there is some \(\mathcal{I}G_3\)-isotopy \(\seq{f_t}{t \in \II}\) such that \(f_1 \pts{\zz} = \wz\). Since \(\mob{\US}\) is a path connected group, in particular \(\mabc{p}{q}{r}^{-1} \circ f_1 \circ \mabc{a}{b}{c}\) lies in \(G_0\) and maps \(\pts{a,b,c,d}\) onto \(\pts{p,q,r,s}\). Therefore, the arc 4-transitivity definition is seen to be satisfied.
	
\subsection{Discussion on Theorem B}

	Let us start by making our terminology precise: relative to four distinguished points \(P = \roost{p_0, \ldots, p_3}\) on the sphere, a closed loop \(\alpha: \II \to \US\) based at \(p_0\) will be referred to as \emph{topological figure 8} if it defines in the fundamental group \(\pi_1 \pts{\US \setminus \roost{p_1, p_2, p_3} ; \,  p_0}\) an element in the same homotopy class of a path of the form \(\bar{\zeta_2} \ast \zeta_1\). Here, each \(\zeta_i\) is a Jordan curve separating \(p_i\) from the remaining two points, while leaving \(p_i\) in the connected component of its complement locally at its left, and \(\ast\), \(\bar{\cdot}\) denote the usual operations of concatenation and inversion. When \(p_3\) is placed at infinity, we may think of the prototypical planar figure 8, consisting of the wedge of two circles based at \(p_0\), each of them traversed once with contrary orientations, and leaving \(p_1\) and \(p_2\) in opposite components of their complements, as pictured in \hyperref[fig:eight]{Figure \ref{fig:eight}}. That is the situation we shall be aiming at.

	\begin{figure}[htb]
		\includegraphics[scale=1.1]{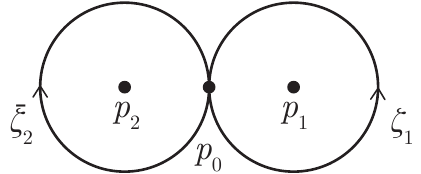}
		\caption{A prototypical topological figure 8 on the plane. A fourth fixed point \(p_3\) is placed at infinity.}
		\label{fig:eight}
	\end{figure}	
	
	As before, a proper extension \(G \subset \diff{1}{\US}\) of \(\mob{\US}\) is fixed throughout. Thus, if \(\chi\) is defined as in \eqref{eq:chi}, we may construct an \(\mathcal{I}G_3\) isotopy \(\seq{K_t}{t\in \RR}\) such that \(\chi \setminus \roost{\inff}\) is realized as the trajectory of \(\zh\) under \(K\). It is given explicitly by
	\begin{equation*}
		K_t = \begin{cases}
			I_{-1-t}^{\zhm} \circ J_{-1} & \, \text{if } t \leq -1 \vg \\[0.25em]
			J_t & \, \text{if } -1 \leq t \leq 1 \vg \\[0.25em]
			I_{-1+t}^{\zhp} \circ J_1 & \, \text{if } t \geq 1 \vg
		\end{cases} 
	\end{equation*}
	where the isotopy \(J\) and the family of points \(\zh\) are described in details in the \hyperref[lemma:crlemma]{Crossing Lemma}, while the isotopies \(I\) are as in the \hyperref[lemma:fundamentallemma]{Fundamental Lemma}. In particular, based on the descriptions given by such lemmas, we obtain a new \(\mathcal{I}G_3\) isotopy \(L_t = \T{\0}{\inff} \circ K_t \circ \T{\0}{\inff}\) and a point \(\yh = \T{\0}{\inff} (\zh) \in \opint{\1,\!\inff}\) such that:
\begin{enumerate}[(i)]
	\item \(\omegalim{\yh}{L} = \alphalim{\yh}{L} = \roost{\0}\),
	\item \(\set{L_t (\yh)}{t \geq 1} \subset \hm\) and \(\set{L_t (\yh)}{t \leq -1} \subset \hp\),
	\item the trajectory \(\orb[L]{\yh}\) only intersects the meridian \(\Gamma\) on the open arc \(\opint{\1, \! \inff}\).
\end{enumerate}

	\begin{auxclaim} \label{lemma:homot}
		There exist a point \(\wh \in \opint{\0, \! \1}\) and an \(\mathcal{I}G_3\) isotopy \(\seq{\varphi_t}{t \in \II}\) such that \(\orb[\varphi]{\wh} \cong \xi_1 \, \rel \roost{\0, \1, \inff}\), where \(\cong\) denotes fixed endpoints homotopy, and \(\xi_1\) is a circle passing through \(\wh\) -- traversed once clockwise while leaving \(\1\) on its right and both \(\0, \inff\) on its left.
	\end{auxclaim}
	\begin{proof}
		Let \(\seq{L_t}{t \in \RR}\) and \(\yh\) be as in the previous discussion, and consider the path \(\lambda (t) = L_t \pts{ \yh }\) describing the trajectory of \(\yh\) under \(L\). Since \(\lambda (0) = \yh \in \Gamma\) and, by item (ii) above, \(\lambda(-1) \in \hp\), we can fix \(	\ttil = \min \lambda^{-1} (\Gamma) > -1\)	and look at the restricted curve \(\tilde{\lambda} = \restric{\lambda}{(-\infty, \ttil]}\). 
		
		On the one hand, item (i) above implies \(\tilde{\lambda} \pts{t_N}\) in the same connected component of \(\compl{\chi}\) as \(\0\), for sufficiently negative \(t_N < 0\). Thus, it can be joined to \(\opint{\0, \!\1}\) by a path not intercepting \(\chi\). On the other hand, \(\tilde{\lambda} \pts{\,\ttil \,} \in \opint{\1, \!\inff}\), by item (iii). It follows from \hyperref[lemma:continuum]{Claim \ref{lemma:continuum}} that \(\tilde{\lambda}\) must intercept \(\chi \setminus \roost{\inff}\) at least once. Therefore, we may define \(	\tm = \max \lambda^{-1} ( \chi ) \cap (-\infty, \ttil] < \ttil\). Then, \(\lambda \pts{\tm} = L_{\tm} \pts{\yh} \in \chi \cap \hp\). Proceeding analogously for times greater than \(\max \lambda^{-1} (\Gamma)\), but using the parts of items (i) and (ii) concerning positive times, we obtain \(\tp > 0\) such that \(\lambda \pts{\tp} = L_{\tp} \pts{\yh} \in \chi \cap \hm\). Also, \(\tilde{\lambda} (t) = L_t \pts{\yh}\) does not intercept \(\chi\) for \(\tm < t < \tp\).
		
		Now, let \(\seq{K_t}{t \in \RR}\) be the isotopy under which the trajectory of \(\zh\) is \(\chi \setminus \roost{\inff}\). This means that \(\lambda \pts{\tm} = L_{\tm} \pts{\yh} = K_a \pts{\zh}\) and that \(\lambda \pts{\tp} = L_{\tp} \pts{\yh} = K_b \pts{\zh}\) , for some \(a,b \in \RR\). In particular, \(K_a \pts{\zh}\) and \(K_b \pts{\zh}\) lie in opposite hemispheres, so there must be an intermediate parameter \(c\) for which \(K_c \pts{\zh} \in \Gamma\). Upon defining \(\wh = K_c \pts{\zh}\), we actually know from the \hyperref[lemma:crlemma]{Crossing Lemma} that \(\wh \in \opint{\0, \! \1}\). We finally define:
		\[
			\varphi_t = \begin{cases}
				K_{c + 3t (b-c)} \circ K_{c}^{-1} & \text{if } 0 \leq t \leq \dfrac{1}{3} \vg \\[0.75em]
				L_{2\tm - \tp + 3t \pts{\tp - \tm}} \circ L^{-1}_{\tm} \circ K_b \circ K_{c}^{-1} & \text{if } \dfrac{1}{3} \leq t \leq \dfrac{2}{3} \vg \\[0.75em]
				K_{3a - 2c + 3t \pts{c - a}} \circ K^{-1}_{a} \circ L_{\tp} \circ L_{\tm}^{-1} \circ K_b \circ K_{c}^{-1} & \text{if } \dfrac{2}{3} \leq t \leq 1 \pt
			\end{cases} 
		\]
		The family \(\pts{\varphi_t}_{t \in \II}\) is readily seen to form an \(\mathcal{I}G_3\) isotopy. Also, \(\varphi_1 \pts{\wh} = \wh\), so the path \(\gamma(s) = \varphi_s \pts{\wh}\) describing the trajectory of the point \(\wh\) is indeed a closed loop based at \(\wh\). Along \(\gamma\), we have four special points, distinguished during the previous constructions:
		\begin{itemize}
			\item the starting and terminal point \(\wh = \gamma (0) = \gamma (1) \in \opint{\0,\!\1}\),
			\item the point \(\gamma(1/3) = K_b \pts{\zh} = L_{\tm} \pts{\yh} \in \hp\),
			\item the point \(\yh \in \opint{\1,\!\inff}\), which is of the form \(\yh = \gamma (\sbar)\), for some \(1/3 < \sbar < 2/3\), and
			\item the point \(\gamma(2/3) = L_{\tp} \pts{\yh} = K_a \pts{\zh} \in \hm\).
		\end{itemize}
		We also know that \(\restric{\gamma}{[0,1/3]}\) and \(\restric{\gamma}{[2/3,1]}\) only intercept \(\Gamma\) at \(\opint{\0,\!\1}\), and that \(\restric{\gamma}{[1/3,\sbar]}\) and \(\restric{\gamma}{[\sbar,2/3]}\) only intercept \(\Gamma\) at \(\opint{\1,\!\inff}\). These facts are enough to show, by combining at most two consecutive straight line homotopies, that \(\gamma\) is homotopic on the plane -- with fixed basepoint \(\wh\) and relative to \(\roost{\0,\1}\) -- to a simple closed curved (say, polygonal) turning once clockwise around \(\1\) and leaving \(\0\) outside of it, as suggested by \hyperref[fig:thmb]{Figure \ref{fig:thmb}}. Moving back to the sphere by adjoining \(\inff\), this amounts to the claimed statement.
	\end{proof}

	\begin{figure}[htb]
		\includegraphics[scale=1.1]{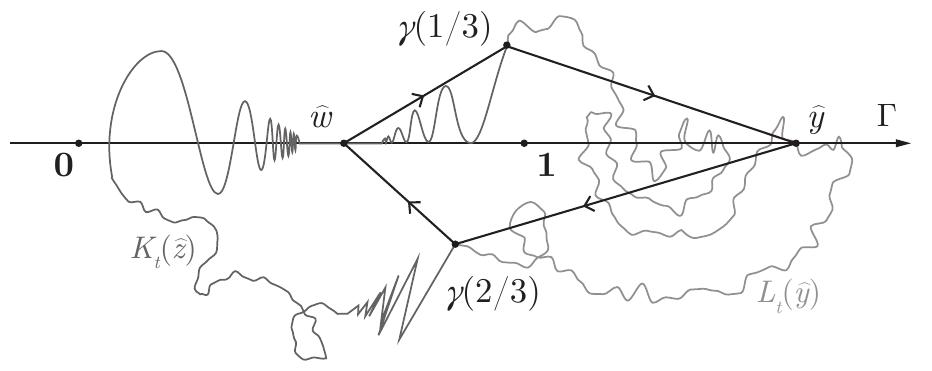}
		\caption{The geometrical information gathered about the paths \(t \mapsto K_t \pts{\zh}\) and \(t \mapsto L_t \pts{\yh}\) is enough to prove that the trajectory of the point \(\wh\) under \(\varphi\) is homotopic to a closed loop turning around \(\1\), but not \(\0\), with fixed endpoint \(\wh\).}
		\label{fig:thmb}
	\end{figure}
	
	We can now complete the construction, for if \(\wh' = \T{\0}{\1} \pts{\wh}\), \hyperref[thm:thma]{Theorem A} yields \(h\) in the identity component of \(G_3\) such that \(\wh' = f\pts{\wh}\). We then consider the following \(\mathcal{I}G_3\) isotopy:
	\[
		\psi_t = h^{-1} \circ \T{\0}{\1} \circ \varphi_t \circ \T{\0}{\1} \circ h \pt
	\]	 
	The trajectory of \(\wh\) under \(\psi\) is described by the curve \(\eta (s) = h^{-1} \pts{ \T{\0}{\1} \circ \gamma (s) }\), where \(\gamma\) describes \(\orb[\varphi]{\wh}\), as in the proof of \hyperref[lemma:homot]{Claim \ref{lemma:homot}}. Since \(\T{\0}{\1}\) leaves \(\roost{\0,\1,\inff}\) invariant, this implies \(\T{\0}{\1} \circ \gamma \cong \T{\0}{\1} \circ \xi_1 \, \rel \roost{\0, \1, \inff}\). But we know that \(\T{\0}{\1} \circ \xi_1\) is also a circle -- traversed once clockwise, while leaving \(\0\) on its right and both \(\1,\inff\) on its left. Thus, since \(h \in G_3\), \(\xi_2 = h^{-1} \circ \T{\0}{\1} \circ \xi_1\) is a clockwise oriented Jordan curve based at \(\wh\), leaving \(\0\) on its right and both \(\1,\inff\) on its left. Also, \(\eta \cong \xi_2 \, \rel \roost{\0, \1, \inff}\). If we now define the \(\mathcal{I}G_3\) isotopy
	\[
		F_t = \begin{cases}
			\varphi_{1-2t} \circ \varphi_{1}^{-1} & \text{if } 0 \leq t \leq \dfrac{1}{2} \vg \\[0.5em]
			\psi_{2t - 1} \circ \varphi_1^{-1} & \text{if } \dfrac{1}{2} \leq t \leq 1 \vg
		\end{cases} 
	\]
	then, clearly 
	\[
		\orb[F]{\wh} = \orb[\psi]{\wh} \ast \overline{\orb[\varphi]{\wh}} = \eta \ast \bar{\gamma} \cong \xi_2 \ast \bar{\xi_1} \; \rel \roost{\0, \1, \inff} \vg
	\]
	meaning that (upon considering \(\zeta_i = \bar{\xi_i}\)) the diffeomorphism \(f \eqdef F_1\) fixes \(P = \roost{\wh, \0, \1, \inff}\), whilst the trajectory \(\orb[F]{\wh}\) of \(\wh\) under the isotopy \(\seq{F_t}{t \in \II}\) is a topological figure 8 relative to this set.
	
	With respect to \(f\), we now evoke the Nielsen-Thurston Classification Theorem, as presented in Section 7.5 of \cite{Franks02}. It takes as input the homeomorphism \(f: \US \to \US\) and the set \(P\), for which the four punctured sphere \(\US \setminus P\) has negative Euler characteristic. Then, it yields 
	\begin{itemize}
		\item a homeomorphism \(\Phi: \US \to \US\) such that \(f\) and \(\Phi\) are isotopic relative to \(P\) or, in other words, such that every point of \(P\) is fixed throughout the isotopy between \(f\) and \(\Phi\),
		\item a (possibly empty) system of closed simple loops \(\alpha_1, \ldots, \alpha_r\), called \emph{reducing curves}.
	\end{itemize}		
	Such curves come equipped with pairwise disjoint tubular neighbourhoods \(V_i\), disjoint from \(P\), such that the connected components of \(\US \setminus \bigcup_{i=1}^{r} V_i\) group into invariant cycles, restricted to which either \(\Phi\) is of \emph{finite order} -- meaning some power of it equals the identity -- or \(\Phi\) is \emph{pseudo-Anosov} relative to \(P\).
	
	Briefly, pseudo-Anosov means that \(\Phi\) (or its restriction to the appropriate component) admits a pair of invariant transverse foliations -- one of which is expanded with ratio \(\beta > 1\) and the other of which is contracted with ratio \(\beta^{-1}\) (in a measure-theoretical precise sense). \emph{Relative to \(P\)} means that the points in \(P\) are kept fixed under \(\Phi\), and manifest as one-prong singularities of the foliations.
	
	If the reducing system of curves is empty, we notice that no power \(f^m\) can be isotopic to the identity relative to \(P\). Indeed, let \(\alpha\) be any simple closed loop separating two pairs of points of \(P\). Then, \(f^m \pts{\alpha}\) would have to be (freely) homotopic to \(\alpha\) relative to \(P\). But, considering the \(\mathcal{I}G_3\) isotopy \(\seq{F^m_t}{t \in \II}\), the family \(F_t^{m} (\alpha)\) would have to enclose the figure 8 in a fashion incompatible with such fact.
	
	Upon interpreting points in \(P\) as punctures (as implied by \cite{Matsuoka05} and Section 7.6 of \cite{Franks02}), each connected component of \(\US \setminus \bigcup_{i=1}^{r} V_i\) must also have negative Euler characteristic. Thus, if the reducing system of curves is nonempty, \(\wh\) must be enclosed by \(\alpha_i\) along with some other reference point \(a \in \roost{\0, \1, \inff}\). Then, the argument from the previous paragraph also shows that \(\Phi\) cannot be of finite order in any component.

	In short, \(\Phi\) must be a pseudo-Anosov map, which is classically known to have strictly positive topological entropy \(\htop{\Phi} = \log \beta > 0\). But the behaviour of mappings isotopic to pseudo-Anosov homeomorphisms is described in \cite{Handel85}. More specifically, Theorem 2 therein implies the statement of \hyperref[thm:thmb]{Theorem B}, and also \(\htop{f} \geq \htop{\Phi}\). Thus, \(f \in G_{0}\) must have positive topological entropy as well.

\section{Characterization of the conformal group in terms of transitivity} \label{sec:appendix}

	In this section, we show that sharp 3-transitivity is a defining property of \(\mob{\US}\) among the homogeneous groups of diffeomorphisms, as more precisely stated in \hyperref[thm:thmc]{Theorem C}. Before proving it, we stablish two auxiliary lemmas. The first shows that the subgroup \(G_2\) fixing the poles possesses a property that we already know \textit{a priori} to be held by the actual \(\mob{\US}\). The second establishes conformality at the poles in this subgroup.

\begin{lemma} \label{lemma:parallels}
	Let \(G \subset \homeo{\US}\) be a sharply 3-transitive homogeneous group. Then, the subgroup \(G_2\) permutes parallels.
\end{lemma}
\begin{proof}
	Recalling that \(G_2\) is the subgroup fixing \(\0\) and \(\inff\), \(g \in G_2\) translates to a planar homeomorphism fixing the origin, for which we must prove that circles centered at the origin are mapped onto circles centered at the origin. Given one such circle \(\gamma\), let \(\lambda = g \pts{\gamma}\) be its image. Then, \(\lambda\) is a Jordan curve, containing the origin in its interior. If \(\lambda\) is \emph{not} a circle, it contains points \(\pmin\) and \(\pmax\) such that
	\[
		\abs{\pmin} = \min_{p \in \lambda} \abs{p} < \max_{p \in \lambda} \abs{p} = \abs{\pmax} \pt
	\]
	For each polar angle \(0 \leq \theta < 2\pi\), the semiradius \(\rad{\theta} = \set{t \e{\ii \theta}}{t \geq 0}\) intercepts \(\lambda\) in a compact set \(\lambda_{\theta}\), in such a way that \(\lambda = \bigcup_{0 \leq \theta < 2\pi} \lambda_{\theta}\). For some \(\theta_0\), it must be the case that \(\abs{p} < \abs{\pmax}\) for every \(p \in \lambda_{\theta_0}\). 
		
	Now, if \(\ttm\) is such that \(\pmax \in \lambda_{\ttm}\), let \(R\) be a planar rotation mapping the semiradius \(\rad{\ttm}\) onto the semiradius \(\rad{\ttz}\). Then, \(R \pts{\pmax} \in R \pts{\lambda}\), but \(R \pts{\pmax} \in \exter \lambda = \compl{\pts{\cl{\inter \lambda}}}\), where \(\inter\) and \(\exter\) are used in the Jordan Curve Theorem sense, as the bounded and unbounded open connected components of \(\compl{\lambda}\), sharing \(\lambda\) as their common boundary. We remark that \(\compl{\pts{ \cl{\inter \lambda} }} \subset \set{z}{\abs{z} > \abs{\pmin}}\). So, it cannot be the case that \(R \pts{\lambda}\) is fully contained within \(\exter \lambda\), since \(\abs{R \pts{\pmin} } = \abs{\pmin}\). Therefore, we also must have \(R \pts{\lambda} \cap \cl{\inter \lambda} \neq \emptyset\). It follows that \(R \pts{\lambda} \cap \lambda \neq \emptyset\).
	
	That said, we may obtain \(p,q \in \gamma\) such that \(R \pts{g(p)} = g(q)\), implying \(\pts{g^{-1} \circ R \circ g}(p) = q\). Since \(p,q\) lie on the same circle, there exists a planar rotation \(U\) such that \(p = U (q)\). But then, \(g^{-1} \circ R \circ g \circ U\) defines an element of \(G_2\) fixing \(q\). By sharp 3-transitivity, it must be the identity. In particular, \(	g^{-1} \circ R \circ g \circ U \; \pts{\gamma} = \gamma\). Since rotations leave \(\gamma\) invariant, the above implies \(R \pts{\lambda} = \lambda\), which is a contradiction. Therefore, \(\lambda\) has to be a circle.
\end{proof}	

\begin{lemma} \label{lemma:conformal}
	Let \(G \subset \homeo{\US}\) be a sharply 3-transitive homogeneous group of diffeomorphisms. Then, every \(g \in G_2\) is conformal at the poles.
\end{lemma}
\begin{proof}
	Given such \(g\), we know from \hyperref[lemma:parallels]{Lemma \ref{lemma:parallels}} that it permutes paralllels. It thus suffices to consider the case of a planar diffeomorphism fixing the origin and mapping circles centered at the origin onto circles centered at the origin. We know that \(A = \dif{g}{\0}\) is a linear isomorphism, so we may fix \(\zmin, \zmax \in \UC\) such that
	\[
		0 < \abs{A \zmin} = \min_{z \in \UC} \abs{Az} \leq \max_{z \in \UC} \abs{Az} = \abs{A \zmax} \pt 
	\]
	Since \(g\) preserves circles, for each \(t \in \opint{0,\!1}\):
	\[
		1 = \frac{ \abs{g\pts{t \zmin}} }{ \abs{g\pts{t \zmax}} } = \frac{ \abs{A \pts{t \zmin} + o \pts{\abs{ t \zmin}}} }{ \abs{A \pts{t \zmax} + o \pts{\abs{t \zmax}}} } = \frac{ \abs{ A \pts{\zmin} + \frac{o \pts{\abs{t \zmin}}}{t}} }{ \abs{ A \pts{\zmax} + \frac{o \pts{\abs{t \zmax}}}{t}} } \to \frac{ \abs{ A \pts{\zmin}} }{ \abs{ A \pts{\zmax}} } \quad \aas t \to 0^{+} \pt
	\]	
	It follows that \(A \pts{\UC}\) is a circle. Since \(A\) is orientation preserving, this is enough to conclude that it is a conformal matrix. In other words, \(g\) is conformal at \(\0\). 
\end{proof}

	Before proceeding to the proof of \hyperref[thm:thmc]{Theorem C}, let us make a small {remark}: if \(G\) is a 2-transitive homogeneous group of diffeomorphisms and \(\delta > 0\) is given, we may find \(\hd \in G\) such that \(\hd\) fixes \(\0\), but not \(\inff\), and \(\dif{\hd}{\0}\) is \(\delta\)-close to \(\id\).
	
	Indeed, this can be verified as follows: by 2-transitivity, we may fix \(h \in G\) such that \(h \pts{\0} = \0\) and \(h \pts{\inff} = \1\). Then, as long as \(0 \leq t < 2\pi\), \(h_t = h^{-1} \circ R_t \circ h\) fixes \(\0\), but not \(\inff\), since \(R_t \pts{\1}\) is a point on the equator distinct from \(\1\). It thus suffices to take \(\hd = h_{\tht}\), for sufficiently small \(\tht\). This might seem like an underuse of the 3-transitivity hypothesis. But, as it turns out, a result from \cite{TalKwakkel14} implies that any 2-transitive homogeneous group must actually be 3-transitive. Either way, we are now ready to establish \hyperref[thm:thmc]{Theorem C}.
	
	\subsection{Proof of Theorem C} 
	
	Suppose, for the sake of contradiction, that \(G\) contains a nonconformal mapping \(g\). By precomposing and postcomposing with suitable rotations, it may be assumed that \(g\) fixes \(\0\) and that \(A = \dif{g}{\0}\) is a nonconformal matrix. This means that, for some pair of unit vectors \(\uz\) and \(\vz\),
	\begin{equation}
		\text{if } \alpha = \iprod{\uz}{\vz} \aand \beta = \iprod{\frac{A \uz}{ \abs{A \uz} }}{\frac{A \uz}{ \abs{A \uz}}} \vg \text{ then } \alpha \neq \beta \pt
	\end{equation}
	Given \(\eps = \abs{\beta - \alpha}/2\) there exists \(\delta > 0\) such that, for every \(u,v,w,z\) unit vectors, \(\abs{u - z} < \delta\) and \(\abs{v - w} < \delta\) imply \(|\iprod{u}{v} - \iprod{z}{w}| < \eps\). We then fix \(\hd \in G\) such that
	\begin{itemize}
		\item \(\hd (\0) = \0\), 
		\item \(\hd^{-1} \pts{\inff} = \pd \neq \inff\), and
		\item \(\Cd = \dif{\hd}{\0}\) is \(\delta/2\)-close to \(\id\).
	\end{itemize}
	Since \(G\) is 3-transitive, we may fix \(\fd \in G_2\) such that \(\fd \pts{ g (\inff) } = \pd\). From \hyperref[lemma:conformal]{Lemma \ref{lemma:conformal}}, we know that \(B = \dif{\fd}{\0}\) is conformal. If we let \(g = \hd \circ \fd \circ g\), then also \(g \in G_2\), and thus \(D = \dif{g}{\0}\) is conformal. But, by the Chain Rule, \(D = \Cd B A\). We remark that, if \(\wz\) is any unit vector, then
	\begin{align*}
		\abs{ \frac{\Cd \wz}{\abs{\Cd \wz}} - \wz} &= \abs{ \Cd \wz - \wz - \frac{\abs{\Cd \wz} - 1}{ \abs{\Cd \wz}} \, \Cd \wz} \\
		&\leq 2 \, \abs{ \Cd \wz - \wz } < \delta \pt
	\end{align*}
	Since \(B\) is conformal, \(\ds \abs{ \frac{D\uz}{\abs{D\uz}} - \frac{A \uz}{\abs{ A \uz}} } < \delta\) and \(\ds\abs{ \frac{D\vz}{\abs{D\vz}} - \frac{A \vz}{\abs{A \vz}} } < \delta\). Accordingly,
	\[
		\abs{ \iprod{ \frac{D\uz}{\abs{D\uz}} }{ \frac{D\vz}{\abs{D\vz}} } - \iprod{ \frac{A \uz}{\abs{ A \uz}} }{ \frac{A \vz}{\abs{A \vz}} } } =  \abs{ \iprod{ \frac{D\uz}{\abs{D\uz}} }{ \frac{D\vz}{\abs{D\vz}} } - \beta }  < \eps \pt
	\]
	By the choice of \(\eps\), this ensures that the angle between \(D \uz\) and \(D \vz\) is different from the angle between \(\uz\) and \(\vz\), contradicting \(D\) conformal. Thus, such nonconformal \(g \in G\) cannot exist, and $G$ is a subgroup of \(\mob{\US}\).
	
	Lastly, let $M\in \mob{\US}$. Then, there exists \(h \in G \subset \mob{\US}\) such that \(h \pts{\0} = M(\0)\), \(h \pts{\1} = M \pts{\1}\) and \(h \pts{\inff} = M \pts{\inff}\). By the sharp 3-transitivity of \(\mob{\US}\), this implies \(M = h \in G\). Since \(M\) was arbitrary, \(G = \mob{\US}\) follows.
	
	\subsection*{Acknowledgements} 
	{
		\small U. Lakatos would like to thank Prof. S. Alvarez at Udelar for his genuine interest and fruitful discussions on the 
obtainment of uniform bounds on the proof of Lemma \ref{lemma:fundamentallemma}.
	}

\setstretch{1.0}

\end{document}